\documentclass[reqno]{amsart}
\usepackage{mathtools,amsthm,amssymb,amsopn}
\usepackage{fullpage}
\usepackage{thm-restate}
\usepackage{anyfontsize}
\usepackage{newtxtext}
\usepackage{newtxmath}
\usepackage{amsmath}
\usepackage{enumerate}
\usepackage{bbm}
\usepackage{thmtools}
\usepackage{spectralsequences} 
\usepackage{tikz-cd}
\usepackage[colorinlistoftodos]{todonotes}
\tikzset{
	symbol/.style={
		draw=none,
		every to/.append style={
			edge node={node [sloped, allow upside down, auto=false]{$#1$}}}
	}
}
\usepackage[colorinlistoftodos]{todonotes}
\usepackage[all]{xy}
\usepackage{amsfonts}
\usepackage[bbgreekl]{mathbbol}
\usepackage{csquotes}
\usepackage[
backend=biber,
style=alphabetic,
sorting=nyt,
giveninits=true,
maxcitenames=4,
maxbibnames=4,
maxnames = 4,
maxalphanames=4
]{biblatex}
\DeclareFieldFormat*{title}{#1}
\usepackage{color}
\usepackage{footmisc}
\usepackage{hyperref}
\usepackage{cleveref}
\usepackage{microtype}
\declaretheorem[name=Theorem,numberwithin=section]{thm}
\declaretheorem[name=Conjecture]{conj}
\declaretheorem[name=Conjecture]{conjj}
\hypersetup{
	colorlinks=true, 
	linktoc=all,     
	linkcolor=blue,  
}
\usepackage{trimclip}

\DeclareSymbolFontAlphabet{\mathbbm}{bbold}
\DeclareSymbolFontAlphabet{\mathbb}{AMSb}%

\newcommand{\C}{\mathbb{C}}
\newcommand{\R}{\mathbb{R}}
\newcommand{\Z}{\mathbb{Z}}

\newcommand{\T}{\mathbb{T}}

\newcommand{\bL}{\underline{L}}

\newcommand{\Cx}{\C^\times}
\newcommand{\pt}{\text{pt}}
\newcommand{\pr}{\operatorname{pr}}
\newcommand{\Hom}{\operatorname{Hom}}

\newcommand{\lt}{\mathfrak{t}}
\newcommand{\lig}{\mathfrak{g}}

\newcommand{\HH}{\mathcal{H}}
\newcommand{\spec}{\text{Spec }}
\newcommand{\hol}{\operatorname{Hol}}
\newcommand{\D}{\mathbb{D}}

\newcommand{\rank}{\operatorname{rank}}
\renewcommand\thmcontinues[1]{Continued}
\renewcommand{\ker}{\operatorname{ker}}
\newcommand{\fz}{\tau}

\newcommand{\F}{\mathbb{F}}

\newcommand{\WH}{W\text{-}\mathrm{Hilb}_{\lt_\C}}
\newcommand{\WHH}{W\text{-}\mathrm{Hilb}_{\lt'_\C}}
\newcommand{\NH}{\mathcal{H}}
\newcommand{\oh}{\mathcal{O}}

\newcommand{\bbL}{\mathbb{L}}

\newcommand{\NHcl}{\operatorname{cl}_{\NH}(A)}

\theoremstyle{plain}
\newtheorem{lem}[thm]{Lemma}
\newtheorem{prop}[thm]{Proposition}
\newtheorem{cor}[thm]{Corollary}

\theoremstyle{definition}

\newtheorem{df}[thm]{Definition}
\newtheorem{exmp}[thm]{Example}
\theoremstyle{remark}
\newtheorem{rem}[thm]{Remark}

\crefname{thm}{theorem}{theorems}
\Crefname{thm}{Theorem}{Theorems}

\crefname{prop}{proposition}{propositions}
\Crefname{prop}{Proposition}{Propositions}

\crefname{conj}{conjecture}{conjectures}
\Crefname{conj}{Conjecture}{Conjectures}

\crefname{conjj}{conjecture}{conjectures}
\Crefname{conjj}{Conjecture}{Conjectures}

\crefname{lem}{lemma}{lemmas}
\Crefname{lem}{Lemma}{Lemmas}

\crefname{cor}{corollary}{corollaries}
\Crefname{cor}{Corollary}{Corollaries}

\crefname{df}{definition}{definitions}
\Crefname{df}{Definition}{Definitions}

\crefname{assumption}{assumption}{assumptions}
\Crefname{assumption}{Assumption}{Assumptions}

\crefname{exmp}{example}{examples}
\Crefname{exmp}{Example}{Examples}

\crefname{rem}{remark}{remarks}
\Crefname{rem}{Remark}{Remarks}

\crefname{note}{note}{notes}
\Crefname{note}{Note}{Notes}

\title{2d Mirrors in nonabelian 3d Mirror Symmetry}
\author{Kifung Chan and Naichung Conan Leung}
\date{}
\begin{document}
\begin{abstract}
We establish a connection between (nonabelian) equivariant 2d mirror symmetry and the geometry of Coulomb branches. In the context of 3d mirror symmetry, a Hamiltonian $G$-manifold $Y$ is expected to determine a complex Lagrangian subvariety $\bbL^G_Y$ of the Coulomb branch. Using transverse Hilbert schemes and nil-Hecke algebras, we develop an algebro-geometric framework for studying Coulomb branches and their Lagrangian subvarieties and formulate criteria for the existence of $\bbL^G_Y$ in terms of equivariant 2d mirror symmetry. We then reinterpret these criteria in terms of Lagrangian displaceability and prove the resulting statements using Lagrangian Floer theory.

\end{abstract}
\maketitle
\section{Introduction}\label{section:introduction}
This paper studies equivariant 2d mirror symmetry and relates it to 3d mirror symmetry and the geometry of Coulomb branches.

Coulomb branches are central objects in 3d mirror symmetry, and they also provide a natural framework for understanding equivariant 2d mirror symmetry. Teleman proposed \cite{tel2014,tel2021} that, if $Y$ is a compact symplectic manifold equipped with a Hamiltonian action of a compact Lie group $G$, then its equivariant quantum cohomology $QH_G^*(Y)$ should define a complex Lagrangian subvariety $\bbL^G_Y$ in the pure-gauge Coulomb branch $C_G$. This construction was developed in \cite{GMP} and generalized in \cite{chan2025quantumcohomologyshiftoperators}. From the perspective of 3d mirror symmetry, $Y$ is naturally interpreted as a 3d A-brane in the Higgs branch $T^*[\pt/G_{\C}]$, while $\bbL^G_Y$ is the corresponding 3d B-brane in the Coulomb branch $C_G$. Thus Teleman's conjecture predicts a correspondence between 3d branes on the Higgs and Coulomb sides.

In this paper, we study the construction of $\bbL_Y^G$ from the perspective of $2d$ mirror symmetry. This naturally
leads us to investigate the geometry of Coulomb branches, and provides new insights into equivariant $2d$ mirror symmetry. In the first half
of the paper, we study the algebraic and geometric structure of
the Coulomb branches and reduce the construction of $\bbL_Y^G$
to several 2d mirror symmetric predictions (\Cref{introconj}). We formulate these predictions in terms of
Lagrangian Floer theory and prove them (\Cref{Introthm1,IntroThm,Introthm3}).

The nonabelian case is the main focus of the paper, where Coulomb branches are more complicated and the mirror construction requires new \emph{nonabelianization} ideas. We begin with the abelian case.

\subsection*{Mirror construction, the abelian case}
Suppose $G=T$ is a compact torus. In this case, the Coulomb branch is the cotangent bundle
\[
C_T\cong T^*\check T_\C
\]
of the complexified dual torus $\check T_\C$. The mirror construction of $\bbL^T_Y$ in this case is also due to Teleman (see also \cite{BDGH} for a physical derivation). Suppose that $Y$ is a symplectic manifold, and let $(Y^\vee,f)$ be its Landau--Ginzburg mirror. Teleman's idea is that a Hamiltonian $T$-action on $Y$ corresponds, on the mirror side, to a holomorphic map
\[
\fz_T:Y^\vee\to \check T_\C.
\]
We call $\fz$ the Teleman map. The Teleman map induces a Lagrangian correspondence from $T^*Y^\vee$ to $T^*\check T_\C$, and $C^T_Y$ is obtained by applying this Lagrangian correspondence to the graph of $df$. Explicitly, $C^T_Y$ is the image
\begin{equation}\label{eq:intro_ZTY}
    Z^T_Y=\left\{
(p,z,a)\in Y^\vee\times \check T_\C\times \lt_\C
:
z=\fz_T(p),\quad
df|p=\fz_T^da|_p
\right\}
\end{equation}
under the projection 
\[
Y^\vee\times \check T_\C\times \lt_\C\to \check T_\C\times \lt_\C\cong T^*\check T_\C.
\]
We remark that $C^T_Y$ is an immersed Lagrangian of $T^*\check T_\C$ when the Lagrangian composition involved is transverse. In general, one should use the derived Lagrangian composition, which defines a $0$-shifted Lagrangian in $T^*\check T_\C$; see \cite{PTVV}. 

\subsection*{Nonabelian Coulomb branches}
In the nonabelian case, the Coulomb branch is no longer a cotangent bundle in general, and its geometry is substantially more complicated. Following \cite{BielawskiFoscolo}, we study it via transverse Hilbert schemes. Let $T\subset G$ be a maximal torus with Lie algebra $\lt$ and Weyl group $W$. In \Cref{section:transverse_W-Hilbert_scheme_and_Nil-Hekce_ring}, we associate to any affine scheme $X$ equipped with a $W$-equivariant morphism to $\lt_\C$ its transverse $W$-Hilbert scheme $\WH(X)$. The following interpretation of $\WH(X)$ is the main result of \Cref{section:transverse_W-Hilbert_scheme_and_Nil-Hekce_ring}.
\begin{thm}[=\Cref{thm:WHilb=nilHecke_closure}]\label{introthm:WHilb=nilclosure}
    If $X=\operatorname{Spec} A$, then $\WH(X)=\operatorname{Spec}\NHcl^W$,
    where $\NHcl$ is the nil-Hecke closure of $A$.
\end{thm}
Here, the nil-Hecke closure of $A$ is the smallest algebra containing $A$ and stable under the nil-Hecke action. See \Cref{section:transverse_W-Hilbert_scheme_and_Nil-Hekce_ring} for more details. The universal property of $\WH(X)$ gives a morphism
\[
\jmath:C_G\to \WH(C_T).
\]
The morphism $\jmath$ need not be an isomorphism, but we show that it is always an open embedding, and we describe its image in \Cref{cor:image_of_jmath}. However, it is sometimes more useful, and more conceptual, to characterize the Coulomb branch by its stability under Hamiltonian reduction, as we now explain.

Suppose that we have a short exact sequence of compact connected Lie groups, and a short exact sequence of their maximal tori,
\begin{align}
    1\to G\to G'\to Q\to 1,\label{introsescompact}\\
    1\to T\to T'\to Q\to 1,\label{introsestor}
\end{align}
where $Q$ is a compact torus. Then $\check Q_\C$ acts on $\check T'_\C$, and this action lifts to a complex Hamiltonian action on $C_{T'}=T^*\check T'_\C$. It can be checked that this induces a complex Hamiltonian action of $\check Q_\C$ on $\WHH(C_{T'})$, together with a natural morphism
\begin{equation}\label{introeq:WHH_to_WH}
    \WHH(C_{T'})/\!/\check Q_\C\to \WH(C_T).
\end{equation}
We denote the moment maps on $C_{T'}$ and $\WHH(C_{T'})$ by $\bbmu$ and $\widetilde{\bbmu}$, respectively. We have the following characterization of Coulomb branches.
\begin{thm}[=\Cref{thm:Coulombbranch=Whilb}]\label{introthm_CG=Whilb}
    For each such $G'$, the morphism \eqref{introeq:WHH_to_WH} is an open embedding. Moreover, the Coulomb branch $C_G\subset \WH(C_T)$ is the intersection of the images of \eqref{introeq:WHH_to_WH}, as $G'$ varies.
\end{thm}
\Cref{introthm_CG=Whilb} can be deduced from the main theorem of Bielawski--Foscolo \cite{BielawskiFoscolo}. Although, as explained in \cite{webster2026geometrycoulombbranches}, the main theorem of \Cite{BielawskiFoscolo} does not hold in full generality, both the statement and the proof remain valid for pure-gauge Coulomb branches. We nevertheless include a proof of \Cref{introthm_CG=Whilb}, both for completeness and because it provides useful insight into the mirror construction of 3d branes.

\subsection*{Mirror construction, the nonabelian case}
Our proposal is that the nonabelian mirror brane $\bbL^G_Y$ should be obtained from the abelian mirror brane $\bbL^T_Y$ via the transverse Hilbert-scheme construction. More precisely, we make the following \emph{nonabelianization} conjecture. 
\begin{conj}\label{introconj}
\begin{enumerate}
    \item The subvariety $\bbL^T_Y\subset C_T$ is $W$-invariant.
    \item The natural morphism $ \WH(\bbL^T_Y)\longrightarrow \bbL^T_Y/W$ is an isomorphism.
    \item $\WH(\bbL^T_Y)$ is contained in $C_G$.
\end{enumerate}
\end{conj}

We give some comments on the conjecture.

Part~\textup{(1)}. The $W$-invariance holds in many examples. If one regards $Y^\vee$ as a moduli space of $T$-invariant Lagrangian branes, then there is a natural $W$-action on these branes. Moreover, if $f$ is defined by counting holomorphic discs, then the Hamiltonian $G$-action implies that $f$ should be $W$-invariant. This would imply the $W$-invariance of $\bbL^T_Y$. However, there are also examples, such as Rietsch's mirror for flag varieties, where $Y^\vee$ does not admit a $W$-action, but $\bbL^T_Y$ is nevertheless $W$-invariant.

Part~\textup{(2)}.
By functoriality of $\WH$, there is an embedding $\WH(\bbL^T_Y)\subset \WH(C_T)$ once we know $\bbL^T_Y$ is $W$-invariant. However, applying $\WH$ may result in a loss of information; for instance, one can have $X\neq\varnothing$ but $\WH(X)=\varnothing$. The condition $\WH(\bbL^T_Y)=\bbL^T_Y/W$ rules this out for $\bbL^T_Y$. Indeed, this is equivalent to saying that the coordinate ring $\C[\bbL^T_Y]$ is equal to its nil-Hecke closure, so no information is lost in the process.

In many cases, including Rietsch's mirror for flag varieties, the coordinate ring $\C[\bbL^T_Y]$ is known to be isomorphic to the equivariant quantum cohomology ring $QH_T(Y)$. It can be shown that $QH_T(Y)$ is always equal to its own nil-Hecke closure, giving a direct verification of the second statement in these cases.

Part~\textup{(3)}.
A conceptual approach to proving Part~\textup{(3)} is to find, for each short exact sequence as in \eqref{introsescompact}, a Hamiltonian $G'$-space $Y'$ such that the image of
\begin{equation}\label{eq:intro_stability_Hamiltoninan_reduction}
    \bbL^{T'}_{Y'}\cap \bbmu^{-1}(0)
\longrightarrow
\bbmu^{-1}(0)/\check Q_\C=C_T
\end{equation}
is equal to $\bbL^T_Y$. This would imply that the image of
\[
\left(\WHH(\bbL^{T'}_{Y'})\cap \widetilde{\bbmu}^{-1}(0)\right)/\check Q_\C
\longrightarrow
\WH(C_T)
\]
is equal to $\WH(\bbL^T_Y)=\bbL^G_Y$. It then follows from \Cref{introthm_CG=Whilb} that $\bbL^G_Y\subset C_G$.

In practice, we will take $Y'=Y$ whenever the $G$-action on $Y$ extends to a $G'$-action. This is the case, for example, when $Y=G/T$ is a flag variety. In general, one may need to take
\[
Y'=(Y\times M)/\Gamma
\]
for some auxiliary symplectic manifold $M$ and finite group $\Gamma$.

We will also deduce the following criterion. 
\begin{lem}\label{introlem:lift}
Suppose that $\bbL^T_Y$ is $W$-invariant and reduced, and that $\bbL^{T^{s_\alpha}}_Y$ is reduced for every root $\alpha$.
\begin{enumerate}[(i)]
    \item The natural morphism $\WH(\bbL^T_Y)\to \bbL^T_Y/W$ is an isomorphism if and only if, for every root $\alpha$ and every $(t,a)\in\bbL^T_Y$ satisfying $\alpha(a)=0$, one has $ s_\alpha(t)=t$.

    \item One has $\WH(\bbL^T_Y)\xrightarrow{\cong}\bbL^T_Y/W\subset C_G$ if and only if, for every root $\alpha$ and every $(t,a)\in\bbL^T_Y$ satisfying $\alpha(a)=0$, one has $\alpha^\vee(t)=1$.
\end{enumerate}
\end{lem}
The reducedness assumptions on $\bbL^{T^{s_\alpha}}_Y$ can be removed by interpreting the equalities appearing in \textup{(i)} and \textup{(ii)} scheme-theoretically; see \Cref{rem:remove_reducedness_assumption}.

\subsection*{SYZ and Lagrangian Floer theory}
In \textit{Mirror Symmetry is T-duality}, Strominger, Yau, and Zaslow \cite{SYZ} proposed that $Y^\vee$ should be constructed as a moduli space of Lagrangian branes $\bL=(L,E)$ in $Y$, where $L$ is a Lagrangian submanifold and $E$ is a flat $\operatorname{U}(1)$-bundle on $L$. This picture was later extended to the Fano setting; see, for example, \cite{AurouxComp,FOOOI}. In this case, $Y^\vee$ is equipped with a holomorphic function $f$, the superpotential, defined by counting holomorphic discs.

In \Cref{section:Teleman_and_f,section:main_theorems,section:proof_of_theorems}, we use the SYZ framework to reinterpret \Cref{introconj} in terms of Lagrangian Floer theory and prove the resulting statements. We fix the following data:
\begin{itemize}
    \item $Y$ is a Hamiltonian $G$-manifold, with moment map $\mu:Y\to\lig^*$.
    
    \item $\widetilde B=B-\kappa$ is a $G$-equivariant $B$-field on $Y$.
    
    \item $\bL=(L,E)$ is a $T$-invariant Lagrangian brane, where $L\subset Y$ is a compact Lagrangian submanifold and $E\to L$ is a line bundle with connection satisfying $F_E=iB|_L$.
\end{itemize}
We impose the standard geometric assumptions required to define Lagrangian Floer theory. The relevant definitions, setup, and conventions are given in \Cref{section:Teleman_and_f,section:main_theorems}, and \Cref{appendix:Floer_theory}.

One should think of $\bL$ as defining a point of $Y^\vee$. We define
\begin{equation*}
df(\bL)
=
\sum_{\beta}
n_\beta
\exp\left(i\int_\beta B\right)
\hol_{\partial\beta}(E)
\T^{\int_\beta\omega}
[\partial\beta]
\in H_1(L;\F),
\end{equation*}
where $\F$ is the universal Novikov field and $\T$ is the Novikov parameter. The sum runs over Maslov-index-two disc classes $\beta\in\pi_2(Y,L)$, and $n_\beta$ denotes the count of holomorphic discs representing $\beta$; see \cite{AurouxComp,FOOOI}. We also define
\begin{equation*}
\fz_T(\bL)
=
\exp(-\mu_T(L))\cdot\HH_T(E)
\in
\exp(\lt^*)\cdot\check T
\cong T_\C^\vee.
\end{equation*}
If $\widetilde B=0$, then $\HH_T(E)$ is the monodromy of $E$ along the $T$-orbits. In general, $\HH_T(E)=\exp\bigl(-i\kappa_T(p)\bigr)\hol_T^p(E)
\in\check T$
is the equivariant monodromy defined in \Cref{section:Teleman_and_f}. We discuss the well-definedness and holomorphicity of $df$ and $\fz_T$ in \Cref{section:Teleman_and_f}. For the SYZ perspective on these constructions, we refer the reader to \cite{AurouxComp,Paper1}.

It is straightforward to check that $df$ and $\fz_T$ are $W$-equivariant, in accordance with Part~\textup{(1)} of \Cref{introconj}. Moreover, the condition $df|_{p}=da|_{p}$ in \eqref{eq:intro_ZTY} for $p=\bL$ is now interpreted as
\[
df(\bL)=o_*(a),
\]
where $a\in H_1(T;\F)$, and $o_*:H_1(T;\F)\to H_1(L;\F)$ is the homomorphism induced by an orbit map. The following theorem is a Floer-theoretic reformulation of Part~\textup{(2)} of \Cref{introlem:lift}.
\begin{thm}[=\Cref{thm2}]\label{IntroThm}
Assume that the minimal Maslov number of $L$ is at least $2$ and that $H^\bullet(L)$ is generated in degree one. For any root $\alpha$, if $df(\bL)\in o_*\bigl(\ker(\alpha)\bigr)$, then $\fz_T(\bL)\in\ker(\alpha^\vee)$.
\end{thm}
The assumptions on the minimal Maslov number and the generation of $H^\bullet(L)$ in degree one are technical hypotheses that ensure sufficient control over holomorphic-disc counts. Under these hypotheses, the condition $df(\bL)=0$ implies that $HF^\bullet(\bL,\bL)\neq 0$. Moreover,
\[
\bigcap_{\alpha\in\Phi(G,T)}\ker(\alpha^\vee)\subset\check T_\C
\]
is equal to the center $Z(\check \check G_\C)$ of the Langlands dual group $\check \check G_\C$. Hence, under these hypotheses, the following theorem is a corollary of \Cref{IntroThm}. We nevertheless prove it without either hypothesis.
\begin{thm}[=\Cref{thm1}]\label{Introthm1}
If $HF^\bullet(\bL,\bL)\neq 0$, then $\fz_T(\bL)\in Z(\check G_\C)$.
\end{thm}
Since the nonvanishing of $HF^\bullet(\bL,\bL)$ implies that $\bL$ is Hamiltonian nondisplaceable, \Cref{Introthm1} can therefore be interpreted as imposing constraints on nondisplaceable Lagrangian branes.

Although the formulation of \Cref{IntroThm,Introthm1} in terms of Part~\textup{(ii)} of \Cref{introlem:lift} is neat, it is easier to prove the corresponding version of Part~\textup{(i)}, obtained by replacing $\ker(\alpha^\vee)$ and $Z(\check G_\C)$ with $\check T_\C^{s_\alpha}$ and $\check T_\C^W$, respectively. One can then deduce \Cref{IntroThm,Introthm1} by proving the version of \Cref{introconj}\textup{(3)} explained below.

Given a short exact sequence \eqref{ses_compact}, we may write
\[
G'=(G\times Q')/\Gamma,
\]
where $Q'$ is a compact torus and $\Gamma$ is a finite subgroup of $Z(G\times Q')$. Choose a compact manifold $Z$ equipped with a free $G\times Q'$-action, and set
\[
Y'=(Y\times T^*Z)/\Gamma.
\]
We can choose $Z$ so that $\pi_1(Z)=\pi_2(Z)=0$.
\begin{thm}[=\Cref{thm3}]\label{Introthm3}
Let $L'=(L\times Z)/\Gamma$. Then there exist an equivariant $B$-field on $Y'$ and a Lagrangian brane $\bL'=(L',E')$ such that $df_{Y'}(\bL')$ is the image of $df_Y(\bL)$ under the natural map
\[
 H_1(L;\F)\to H_1(L;\F)\oplus H_1(Z;\F)\to H_1(L\times Z;\F)\to H_1(L';\F),
\]
and $\fz_{T}(\bL)$ is the image of $\fz_{T'}(\bL')$ under the quotient map
\[
 \check \rho:\check T'_\C\to \check T_\C.
\]
\end{thm}
\Cref{Introthm3} is the Floer-theoretic formulation of \eqref{eq:intro_stability_Hamiltoninan_reduction} and therefore addresses \Cref{introconj}\textup{(3)}.

\section*{Acknowledgements}
The authors thank Kwokwai Chan, Justin Hilburn, Chin Hang Eddie Lam, Siu-Cheong Lau, Yan-Lung Leon Li, Cheuk Yu Mak, Michael McBreen, Ziming Nikolas Ma, Tao Su, Yat Hin Suen, Ju Tan, and Tudor Dimofte for valuable discussions at various stages of this project. The work of N. Leung described in this paper was substantially supported by grants from the Research Grants Council of the Hong Kong Special Administrative Region, China (Project Nos. CUHK14301721, CUHK14306322, CUHK14305923 and CUHK14302224) and a direct grant from CUHK.

\section*{Notations}
We use the following notations throughout the paper.
\begin{itemize}
    \item $G$ is a connected compact Lie group with Lie algebra $\lig$, and $T\subset G$ is a maximal torus with Lie algebra $\lt$.
    \item $N(T)$ is the normalizer of $T$ in $G$, and $W=N(T)/T$ is the Weyl group.
    \item Let $\Phi=\Phi(G,T) \subset i\mathfrak{t}^*$ be the root system of $(G,T)$, and fix a system of positive roots $\Phi^+ \subset \Phi$ with simple positive roots $\alpha_1,\dots,\alpha_r$. The corresponding simple reflections are denoted by $s_1,\dots,s_r$.
    \item $\{\sigma_1,\dots,\sigma_m\}$ is a set of homogeneous algebraically independent generators for $\C[\lt_\C]^W$.
    \item $R$ is a $W$-invariant graded subspace of $\C[\lt_\C]$ such that the multiplication map $\C[\lt_\C]^W \otimes R \to \C[\lt_\C]$ is an isomorphism. In particular, $R$ is isomorphic to the regular representation of $W$.
    \item $\Lambda$ is the cocharacter lattice of $T$.
\end{itemize}

\section{The transverse \texorpdfstring{$W$}{W}-Hilbert scheme and Nil-Hecke ring}\label{section:transverse_W-Hilbert_scheme_and_Nil-Hekce_ring} 

Let $X=\spec A$ be an affine $\C$-scheme of finite type equipped with
a $W$-action, and let $\eta:X\to \lt_\C$ be a $W$-equivariant morphism.
\begin{df}
    The \emph{transverse $W$-Hilbert scheme} of $X$, denoted by $\WH(X)$, is the moduli functor that assigns to each $\mathbb C$-scheme $S$, on which $W$ acts
    trivially, the set of $W$-invariant closed subschemes 
    \[
    Z\subset X\times S
    \]
    satisfying the \emph{transverse condition}, i.e., the natural morphism
    \[
    R\otimes\oh_S \longrightarrow p_{S*}\oh_Z
    \]
    is an isomorphism. 
    
    Here elements of $R$ are regarded as functions on $\lt_\C$, and hence as functions on $Z$ via the composition
    \[
    Z\subset X\times S\xrightarrow{\eta_S} \lt_\C\times S \to \lt_\C,
    \]
    and $p_S:X\times S\to S$ denotes the projection.
\end{df}
We note that there is a natural morphism $\WH(X)\to X/W$. Indeed, let $S$ be a scheme, and let $S\to \WH(X)$ be a morphism corresponding to a closed subscheme $Z\subset X\times S$. Composing the inclusion of $Z$ with the projection $X\times S\to X$ gives a $W$-equivariant morphism $Z\to X$. Passing to the quotients by $W$, we obtain
\[
S\cong Z/W\longrightarrow X/W.
\]
This also shows that 
\begin{equation}\label{eq:Z_in_fibre_product}
    Z\subset X\times_{X/S}S.
\end{equation}

\begin{lem}\label{lem:equivalent_condition_transverse}
The transverse condition on a closed subscheme $Z\subset X\times S$ is equivalent to
\begin{enumerate}[(i)]
\item $p_{S*} \mathcal{O}_Z$ is locally isomorphic to $R\otimes\mathcal{O}_S$ as $W$-equivariant $\oh_S$-modules; and
\item the map $\eta_S$ restricts to an isomorphism from $Z$ onto its scheme-theoretic image in $S \times \mathfrak{t}_\C$.
\end{enumerate}
\end{lem}
\begin{proof}
Suppose first that $Z$ satisfies the transverse condition. Then
condition~(i) holds by definition. Moreover, since $\oh_Z$ is
generated as an $\oh_S$-algebra by $\oh_S$ and the image
of
\[
    R\subset\C[\lt_\C],
\]
the natural morphism
\[
    \varphi:
    \C[\lt_\C]\otimes\oh_S
    \longrightarrow
    p_{S*}\oh_Z
\]
is surjective. Hence the restriction of $\eta_S$ to $Z$ is a closed
immersion, and therefore an isomorphism onto its scheme-theoretic
image. Thus condition~(ii) holds.

Conversely, suppose that conditions~(i) and~(ii) hold. We first note that the natural morphism
\begin{equation*}
    \oh_S\to (p_{S*}\oh_Z)^W
\end{equation*}
is an isomorphism. To see this, it suffices to check the assertion locally on $S$.
By condition~\textup{(i)}, we then have locally
\[
(p_{S*}\oh_Z)^W\cong (R\otimes\oh_S)^W=\oh_S.
\]
As a result, for each $f\in \C[\lt_\C]^W$, there exists some $a_f\in\oh_S$ such that $\varphi(f)=\varphi(a_f)$. Since $\varphi$ is surjective by condition~(ii), the decomposition
\[
    \C[\lt_\C]
    \cong
    \C[\lt_\C]^W\otimes R
\]
implies
\[
p_{S*\oh_Z}=\varphi(\C[\lt_\C]\otimes\oh_S)=\varphi(R\otimes\oh_S).
\]
In other words, the morphism
\begin{equation}\label{eq:surjectivefromR}
    \oh_S\otimes R
    \longrightarrow
    p_{S*}\oh_Z
\end{equation}
is surjective. By condition (i), $p_{S*}\oh_Z$ is a locally free sheaf over $S$ of rank $\dim R$. Hence \eqref{eq:surjectivefromR} must be an isomorphism, proving the transverse
condition.
\end{proof}

\begin{exmp}\label{exmp:WHil_Vxt}
    Let $V$ be a finite-dimensional $W$-representation, and let $X=\lt_\C\times V$ with $\eta:X\to \lt_\C$ the first projection. Let 
    \[
    \mathcal{Z}=\lt_\C\times (R\otimes V)^W.
    \]
    We will show that $\WH(X)$ is represented by
    \[
    \mathcal{Z}/W=\lt_\C/W\times(R\otimes V)^W.
    \]
    First note that $\mathcal{Z}$ defines a family of transverse
    subschemes over $\mathcal{Z}/W$. Indeed, there is a $W$-equivariant morphism
    \[
    \mathcal{Z}\longrightarrow \lt_\C\times V,\qquad (t,f)\longmapsto (t,f(t)),
    \]
    where an element $f\in (R\otimes V)^W\subset (\C[\lt_\C]\otimes V)^W$
    is regarded as a $W$-equivariant morphism $f:\lt_\C\to V$. 
    
    Let $S$ be a scheme, and let $S\to \WH(X)$ be an $S$-point corresponding to a $W$-invariant closed subscheme
    \[
    Z\subset \lt_\C\times V\times S.
    \]
    We claim that there exists a unique morphism $\varphi:S\to \mathcal{Z}/W$ such that $Z\cong \varphi^{-1}(\mathcal{Z})$. 
    
    Since the assertion is local on $S$, we may assume that $S$ is affine. Let $I\subset \C[\lt_\C\times V\times  S]$ be the ideal of $Z$. The transverse condition says that for each $x\in V^*\subset \C[V]$, there is a unique $x'\in R\otimes\mathcal{O}_S$ such that $x-x'\in I$. The assignment $x\mapsto x'$ defines an element of
    \[
    \operatorname{Hom}_W(V^*,R\otimes\mathcal{O}_S)
    =
     (R\otimes V)^W\otimes\mathcal{O}_S,
    \]
    and hence gives a morphism $\varphi_2:S\to (R\otimes V)^W$. On the other hand, the projection $Z\to \lt_\C$ descends to a morphism $\varphi_1:S=Z/W\to \lt_\C/W$. Now it is easy to see that $ \varphi=(\varphi_1,\varphi_2)$ is the required morphism.
\end{exmp}

It is easy to deduce the representability of $\WH(X)$ from the representability of the ordinary Hilbert scheme \cite{GrothendieckHilbertSchemes,nitsure2005constructionhilbertquotschemes}. We will also give an alternative construction of $\WH(X)$ via nil-Hecke rings.
\begin{df}
Let $\NH$ be the subalgebra of $\operatorname{End}_{\C}(\C[\lt_\C])$ generated by $\C[\lt_\C]$ (acting by multiplication) and the divided-difference operators
\[
D_i=\frac{1}{\alpha_i}(1-s_i),
\qquad i=1,\dots,r.
\]
We call $\NH$ the nil-Hecke ring associated to $G$.

An $\NH$-algebra is a commutative $\C$-algebra $B$ equipped with an
$\NH$-module structure such that:
\begin{enumerate}
    \item the homomorphism $\C[\lt_\C]\to B$ given by $f\mapsto f\cdot 1$ is a ring homomorphism;
    \item the actions of $D_i$ satisfy the twisted Leibniz rule
    \[
    D_i(bb')
    =
    D_i(b)s_i(b')+bD_i(b'),
    \qquad b,b'\in B,
    \]
    where
    \[
    s_i=1-\alpha_iD_i
    \]
    is the induced action of the simple reflection $s_i$ on $B$.
\end{enumerate}
An $\NH$-algebra homomorphism between two $\NH$-algebras is a ring homomorphism intertwining the $\NH$-action.
\end{df}
\begin{rem}
    The twisted Leibniz rule implies that $W\subset\NH$ acts on $B$ by algebra automorphisms. 
\end{rem}

\begin{df} 
The \emph{nil-Hecke closure} of $A$ is an $\NH$-algebra $B$ equipped with a $W$-equivariant $\C[\lt_\C]$-algebra homomorphism $\psi : A \to B$ satisfying the following universal property: for any commutative $\NH$-algebra $B'$ and any $W$-equivariant $\C[\lt_\C]$-algebra homomorphism $\gamma : A \to B'$, there exists a unique $\NH$-algebra homomorphism $\gamma' : B \to B'$ such that $\gamma = \gamma' \circ \psi$. 
\end{df} 

\begin{thm}\label{thm:WHilb=nilHecke_closure} 
The functor $\WH(X)$ is represented by \[
\spec (\NHcl)^W,
\] 
and the universal family is $\spec (\NHcl)$.
\end{thm}

\begin{lem}\label{GinLemma}
Let $M$ be a $\C[\lt_\C]\rtimes W$-module. The following are equivalent:
\begin{enumerate}[(i)]
\item The natural map
\begin{equation*}
  \C[\lt_\C] \otimes_{\C[\lt_\C]^W}M^W\longrightarrow M
\end{equation*}
is an isomorphism.
\item The $\C[\lt_\C] \rtimes W$-module structure on $M$ extends uniquely to an $\NH$-module structure.
\end{enumerate}

In particular, a $W$-equivariant $\C[\lt_\C]$-algebra $A$ admits a unique $\NH$-algebra structure if and only if the natural map
\begin{equation*}
\theta:\C[\lt_\C]\otimes_{\C[\lt_\C]^W} A^W  \longrightarrow A
\end{equation*}
is an isomorphism.
\end{lem}

\begin{proof}
This is \cite[Lemma~7.15]{Ginzburg}. We also give an alternative proof here. First, (i) clearly implies (ii), because $D_i$ acts on $M^W \otimes_{\C[\lt_\C]^W} \C[\lt_\C]$ by the formula $D_i(p\otimes m)=D_i(p)\otimes m$.

We next claim that for any nonzero $\NH$-module $N$, we have $N^W\neq 0$. To see this, suppose that $N^W=0$, and let $x\in N$ be nonzero. Since $x$ is not $W$-invariant, there exists a simple reflection $s_{i_1}$ such that $s_{i_1}(x)\neq x$. This implies $D_{i_1}(x)\neq 0$, because $x-s_{i_1}(x)=\alpha_{i_1}D_{i_1}(x)$. Replacing $x$ by $D_{i_1}(x)$ and repeating the same argument, we find simple reflections $s_{i_1},\dots,s_{i_n}$ with $n>|\Phi^+|$ such that $D_{i_n}\cdots D_{i_1}(x)\neq 0$. This contradicts the nil-Coxeter relations, according to which any
product
\[
D_{i_n}\cdots D_{i_1}
\]
with $n>|\Phi^+|$ is zero; see, for example, \cite[Theorem~3.4]{BGG}.

Now suppose that (ii) holds, and let $K$ and $C$ be defined by the exact sequence of $\NH$-modules
\begin{equation*}
0\to K\to \C[\lt_\C] \otimes_{\C[\lt_\C]^W}M^W \to M\to C\to 0.
\end{equation*}
Since the functor $N\mapsto N^W$ is exact, we also have the exact sequence 
\begin{equation*}
0\to K^W\to (\C[\lt_\C] \otimes_{\C[\lt_\C]^W}M^W)^W\to M^W\to C^W\to 0.
\end{equation*}
Since
\begin{equation*}
(\C[\lt_\C] \otimes_{\C[\lt_\C]^W}M^W)^W=\C[\lt_\C]^W \otimes_{\C[\lt_\C]^W}M^W =M^W,
\end{equation*}
we have $K^W=C^W=0$, and hence $K=C=0$ by the claim.
\end{proof}
\begin{exmp}\label{exmp:NH_action_on_HG}
    Let $M$ be a topological space equipped with a $G$-action. Then the natural $\C[\lt_\C]\rtimes W$-action on $H_T^\bullet(M)$ extends to an $\NH$-action. This follows from \Cref{GinLemma} and the well-known isomorphism (see, for example, Proposition~1(ii) in \Cite{MR1649623})
    \[
    \C[\lt_\C]\otimes_{\C[\lt_\C]^W}H_G^\bullet(M)\cong H_T^\bullet(M).
    \]
\end{exmp}
\begin{proof}[Proof of \Cref{thm:WHilb=nilHecke_closure}]
First suppose that $\NHcl$ exists. Let $S$ be a scheme. We will show that a morphism 
\[
S\to \WH(X)
\]
corresponds uniquely to a morphism 
\[
S\to \operatorname{Spec}(\NHcl^W).
\]
In view of the uniqueness, we can assume $S=\spec B_0$ is affine. We will show that the following data are equivalent (up to isomorphism):
\begin{enumerate}[(i)]
    \item A morphism $S\to\WH(X)$.
    \item A closed subscheme $Z\subset X\times S$ satisfying the transverse condition.
    \item A pair $(B,\gamma)$, where $B$ is an $\NH$-algebra with $B_0=B^W$, and $\gamma:A\to B$ is a $W$-equivariant $\C[\lt_\C]$-algebra homomorphism.
    \item A $\C$-algebra homomorphism $\NHcl^W\to B_0$.
\end{enumerate}
(i) and (ii) are equivalent by definition. We now show that (ii) and (iii) are equivalent. 

Suppose we have (ii). Since $Z$ is a closed subscheme of the affine scheme $X\times S$, $Z$ is affine and we may write $Z=\spec B$. The restriction of the projection $X\times S\to X$ to $Z$ defines a ring homomorphism $\gamma:A\to B$. Moreover, the transverse condition implies $B^W=B_0$. This gives (iii).

Conversely, suppose we have (iii). Let $i:B_0\to B$ be the inclusion. Then the ring homomorphism
\[
i\otimes \gamma:B_0\otimes A\to B 
\]
defines a morphism
\[
Z\to X\times S,
\]
which is a closed embedding and satisfies the transverse condition by \Cref{GinLemma}. It is easy to see that the two processes above are inverse of each other, so (ii) and (iii) are equivalent.

We next show that (iii) and (iv) are equivalent. Suppose we have (iii). By the universal property of nil-Hecke closure, there is a unique $\NH$-algebra homomorphism $\gamma':\NHcl\to B$ with $\gamma=\gamma'\circ\psi$. Taking $W$-invariants, this gives the $\C$-algebra homomorphism $\NHcl^W\to B_0$. Moreover, we have
\[
\NHcl\cong \NHcl^W\otimes R,\qquad B\cong B_0\otimes R,
\]
which implies
\[
B\cong \NHcl\otimes_{\NHcl^W} B_0.
\]
This shows that $(B,\gamma)$ is determined, uniquely up to isomorphism, by the $\C$-algebra homomorphism $\NHcl^W\to B_0$. This shows (iii) and (iv) are equivalent.

As a result, $\WH(X)=\spec (\NHcl^W)$, and $\spec\NHcl$ is the universal family. Conversely, suppose $\WH(X)$ is represented by an affine scheme $\operatorname{Spec} B_0$ with universal family $\spec B$, then the above also shows that $B=\NHcl$.

It remains to show that $\NHcl$ exists. First note that if $A'$ is another $W$-equivariant $\C[\lt_\C]$-algebra for which $\operatorname{cl}_{\NH}(A')$ exists, and if $A\cong A'/I$ for some $W$-invariant ideal $I\subset A'$, then $\NHcl$ exists and is isomorphic to $\operatorname{cl}_{\NH}(A')/\NH\cdot I$, where $\NH\cdot I$ denotes the $\NH$-invariant ideal generated by the image of $I$. Moreover, since we can embed $X$ as a $W$-invariant closed subscheme of $\lt_\C\times V$ for some finite-dimensional $W$-representation $V$, we can reduce to the case where $X=\lt_\C\times V$. This case follows from \Cref{exmp:WHil_Vxt}.
\end{proof}

\begin{lem}\label{lem:smoothness_WHilb}
    If $X$ is smooth over $\lt_\C$, then $\WH(X)$ is smooth over $\lt_\C/W$. 
\end{lem}
\begin{proof}
Let $p\in \WH(X)$ be a closed point corresponding to a $W$-invariant closed subscheme $Z\subset X$. We show that $\WH(X)$ is smooth over $\lt_\C/W$ at $p$.

Choose a point $x\in Z$, and set $a=\eta(x)$. Let $W_x$ and $W_a$ denote the stabilizers of $x$ and $a$, respectively. The transverse condition implies that $\eta|_Z\to \lt_\C$ is a closed embedding, and in particular it is injective on closed points. It follows that $W_x=W_a$.

Put $W'=W_x=W_a$. The formal neighborhood of $p$ in $\WH(X)$ depends only on the formal neighborhood $\widehat X_Z$ of the embedding $Z\subset X$. Since $Z$ is supported on the $W$-orbit of $x$, we have a decomposition
\[
\widehat X_Z=\bigsqcup_{\sigma\in W/W'}\sigma\widehat X_x.
\]
For any scheme $S$, a $W$-invariant closed subscheme $\mathcal Z\subset \widehat X_Z\times S$ is equivalent to a $W'$-invariant closed subscheme $\mathcal Z_x\subset \widehat X_x\times S$
by the formula $\mathcal Z=\bigcup_{\sigma\in W/W'}\sigma(\mathcal Z_x)$. Moreover, $\mathcal Z$ is transverse if and only if $\mathcal Z_x$ is transverse. Therefore the formal neighborhood of $p$ in $\WH(X)$ is isomorphic to the formal neighborhood of $p_x=[Z_x]$ in the transverse $W'$-Hilbert scheme of $\widehat X_x$.

Thus, replacing $W$ by $W'$, we reduce to the case where $x$ is fixed by $W$. Since $X$ is smooth over $\lt_\C$, there exists a finite-dimensional $W$-representation $V$ such that $\widehat X_x$ is $W$-equivariantly isomorphic to the formal neighborhood of $(a,0)$ in $\lt_\C\times V$, and under this isomorphism the map $\eta$ corresponds to the projection $\lt_\C\times V\to \lt_\C$. Hence the desired smoothness follows from \Cref{exmp:WHil_Vxt}. 
\end{proof}

\section{The pure-gauge Coulomb branch}\label{section:pure_gauge_Coulomb_branch}
Consider the space of polynomial loops in $G$ based at the identity:
\[
\Omega G=\{\gamma:S^1\to G: \gamma\text{ is a polynomial map, and }\gamma(1)=e\}.
\]
The Pontryagin product defines a commutative algebra structure on
$H_\bullet^G(\Omega G)$; see \cite{BFM,BFN} for further details. We write 
\[
C_G=\spec H_\bullet^G(\Omega G).
\]
\begin{exmp}
    Suppose $G=T$ is a compact torus. Let $\Lambda$ be the cocharacter lattice of $T$. Then $\Omega T$ is homotopy equivalent to the discrete set $\Lambda$. Hence, we have
    \[
    H_\bullet^T(\Omega T)\cong H_\bullet^T(\Lambda)=\bigoplus_{\lambda\in \Lambda}\C[\lt_\C]z^\lambda.
    \]
    Moreover, we have $z^\lambda\cdot z^{\lambda'}=z^{\lambda+\lambda'}$. As a result, we have $C_T\cong\check T_\C\times \lt_\C\cong T^*\check T_\C$.
\end{exmp}
\begin{exmp}\label{exmp:SO(3)_and_SU(2)_CG}
In general, by localization we can embed $H_\bullet^T(\Omega G)$ into the field of fractions of $H_\bullet^T(\Omega T)$.

If $G=\operatorname{SO}(3)$, then $H_\bullet^T(\Omega T)=\C[a,z^{\pm 1}]$. Using the fact that $H_\bullet^T(\Omega \operatorname{SO}(3))$ is generated by the Schubert classes corresponding to $\pm 1\in \Z\cong \Lambda$, one obtains
\begin{equation*}
H_\bullet^T(\Omega \operatorname{SO}(3))
=
\C\left[a,z^{\pm 1},\frac{z-z^{-1}}{a}\right],\qquad H_\bullet^G(\Omega \operatorname{SO}(3))
=
\C\left[a,z^{\pm 1},\frac{z-z^{-1}}{a}\right]^W.
\end{equation*}
Here, $W=\Z_2$ acts by sending $z$ to $z^{-1}$ and $a$ to $-a$. Similarly, if $G=\operatorname{SU}(2)$, then $H_\bullet^T(\Omega T)=\C[a,z^{\pm 1}]$, and
\begin{equation*}
H_\bullet^T(\Omega \operatorname{SU}(2))
=
\C\left[a,\frac{z-z^{-1}}{a},\frac{z+z^{-1}-2}{a^2}\right],\qquad H_\bullet^G(\Omega \operatorname{SU}(2))
=
\C\left[a,\frac{z-z^{-1}}{a},\frac{z+z^{-1}-2}{a^2}\right]^W.
\end{equation*}
\end{exmp}
We consider the diagram
\begin{equation}\label{Couldiag1}
\begin{tikzcd}
H_\bullet^T(\Omega G)&H_\bullet^T(\Omega T)\ar[l]\\
H_\bullet^G(\Omega G)\ar[u],
\end{tikzcd}
\end{equation}
where the upper arrow is induced by the pushforward $\Omega T\to \Omega G$, and the vertical arrow is the natural inclusion
$H_\bullet^G(\Omega G)\cong H_\bullet^T(\Omega G)^W\subset H_\bullet^T(\Omega G)$.
By \Cref{exmp:NH_action_on_HG}, the $\C[\lt_\C]\rtimes W$-module structure on $H_\bullet^T(\Omega G)$ extends to an $\NH$-module structure. Hence the diagram induces a morphism
\begin{equation}\label{eq_Coul_to_WHilb}
C_G\to \WH(C_T)
\end{equation}
by \Cref{thm:WHilb=nilHecke_closure}.

\begin{rem}
If $N$ is a representation of $G$, then the corresponding Coulomb branch $C_{G,N}$ is defined in \cite{BFN} by replacing $\Omega G$ with a certain space called the variety of triples. By a diagram similar to \eqref{Couldiag1}, one can show that there is a morphism $C_{G,N}\to \WH(C_{T,N})$.
\end{rem}

The morphism \eqref{eq_Coul_to_WHilb} is in general not an isomorphism. Nevertheless, its image can be characterized using the Hamiltonian
reductions discussed below.

\subsection*{Hamiltonian reduction and extension by a torus}
Suppose that we have short exact sequences
\begin{equation}\label{ses_compact}
    1\to G\to G'\to Q\to 1
\end{equation}
and
\begin{equation}\label{ses_torus}
    1\to T\to T'\to Q\to 1 
\end{equation}
of compact connected Lie groups and of their maximal tori,
respectively, where $Q$ is a compact torus. There is a complex Hamiltonian action of $\check Q_\C$ on $C_{T'}=T^*\check T'_\C$, with moment map
\[
\bbmu:C_{T'}\cong \check T'_\C\times \lt'_\C\xrightarrow{\pr} \lt'_\C\to \mathfrak{q}_\C (\cong \operatorname{Lie}(\check Q_\C)^*),
\]
where $\pr$ is the projection and the last map is
induced on Lie algebras by the homomorphism $T'\to Q$. It is clear that $\bbmu$ is $W$-invariant, and hence induces a morphism
\[
\tilde{\bbmu}:\WHH(C_{T'})\to \mathfrak{q}_\C.
\]
As usual, we define
\begin{equation*}
    \C_{T'}/\!/\check Q_\C\triangleq \bbmu^{-1}(0)/\check Q_\C\cong C_{T},\qquad \WHH(C_{T'})/\!/\check Q_\C\triangleq ({\widetilde{\bbmu}})^{-1}(0)/\check Q_\C,
\end{equation*}
where the quotients are understood as affine quotients.

\begin{lem}
    There is a natural morphism
    \[
    \WHH(C_{T'})/\!/\check Q_\C\to \WH(C_{T}).
    \]
\end{lem}

\begin{proof}
    Let $S$ be a scheme, and let $S\to {\widetilde{\bbmu}}^{-1}(0)\subset \WHH(C_{T'})$ be a morphism, and let $Z\subset C_{T'}\times S$ be the corresponding closed
    subscheme. Since the pullback of $\widetilde{\bbmu}$ via $Z\to S\to \WH(C_{T'})$ is equal to the pullback of $\bbmu$ via $Z\to C_{T'}$, we have
    \[
    Z\subset \bbmu^{-1}(0)\times S=\check T'_\C\times \lt_\C.
    \]    
    Note that the composition
    \[
    Z\to  \bbmu^{-1}(0)\times S\to  C_T\times S
    \]
    is also a closed embedding by the transverse condition, and hence defines an element of $\WH(T^*\check T_\C)(S)$. This construction is functorial in $S$ and
    $\check Q_{\C}$-invariant. Hence it descends to the affine
    quotient and induces a morphism
    \[
    \WHH(T^*\check T'_\C)/\!/\check Q_\C\to \WH(T^*\check T_\C),
    \]
    as desired.
\end{proof}

The following theorem can be obtained from the main theorem of \cite{BielawskiFoscolo} in the pure-gauge case. As explained in \cite{webster2026geometrycoulombbranches}, the main
theorem of \emph{loc.\ cit.} does not hold in full generality;
however, both its statement and proof remain valid in the present
setting. For completeness, we provide a proof in \Cref{secion:proof_of_Coulombbranch=Whilb}.
\begin{thm}\label{thm:Coulombbranch=Whilb}
For every torus extension
\[
    1\longrightarrow G\longrightarrow G'
    \longrightarrow Q\longrightarrow 1,
\]
the natural morphism
\[
    \WHH(C_{T'})/\!/\check Q_\C\to \WH(C_{T})
\]
is an open embedding. Moreover, the morphism
\[
    C_G\longrightarrow
    \WH(C_{T})
\]
identifies $C_G$ with the scheme-theoretic intersection of the images
of these open embeddings, as $G'$ ranges over all such torus
extensions.
\end{thm}

\begin{exmp}
When $G=\operatorname{SO}(3)$ or $G=\operatorname{SU}(2)$, we have $C_T=T^*\Cx$, and (cf. \Cref{thm:WHilb=nilHecke_closure})
\[
\WH(C_T)=\operatorname{Spec}\C\left[a,z^{\pm 1},\frac{z-z^{-1}}{a}\right]^{W}.
\]
This scheme agrees with $C_G$ when $G=\mathrm{SO}(3)$, but not
when $G=\mathrm{SU}(2)$, as seen in \Cref{exmp:SO(3)_and_SU(2)_CG}.

Note that there is a short exact sequence
\[
1\to \operatorname{SU}(2)\to \operatorname{U}(2)\to \operatorname{U}(1)\to 1.
\]
We take $G=\operatorname{SU}(2)$, $G'=\operatorname{U}(2)$, and $Q=\operatorname{U}(1)$. Then
\[
\WH(C_{T'})
=
\operatorname{Spec}
\C\left[
a_1,a_2,z_1^{\pm 1},z_2^{\pm 1},
\frac{z_1-z_2}{a_1-a_2}
\right]^W.
\]
The action of $\check Q_\C\cong \Cx$ on $\WH(C_{T'})$ corresponds to the $\Z$-grading on the coordinate ring given by $\deg(z_i)=1$ and $\deg(a_i)=0$. The subscheme $({\widetilde{\bbmu}})^{-1}(0)$ is obtained by imposing $a_1=-a_2=a$. A direct computation shows that the Hamiltonian reduction is
\[
({\widetilde{\bbmu}})^{-1}(0)/\!/\check Q
=
\operatorname{Spec}
\C\left[
a,
\frac{z-z^{-1}}{a},
\frac{z+z^{-1}-2}{a^2}
\right]^W,
\]
where $z=z_1z_2^{-1}$. This is isomorphic to $C_{\operatorname{SU}(2)}$.
\end{exmp}

\section{Proof of Theorem~\ref{thm:Coulombbranch=Whilb}}\label{secion:proof_of_Coulombbranch=Whilb}
\begin{lem}\label{Liegroup_lemma}
Suppose that $G$ is a compact connected Lie group with no direct factor isomorphic to $\operatorname{Sp}(n)$, and let $\alpha$ be a root of $(G,T)$. Then there exists a cocharacter $\lambda$ of $T$ such that $\langle\alpha,\lambda\rangle=1$.
\end{lem}

\begin{proof}
We first explain why the lemma fails when $G=\operatorname{Sp}(n)$. In this case, the root system can be embedded into $\Z^n$ with standard basis $e_1,\dots,e_n$ such that 
\[ 
\Phi=\{\pm e_i\pm e_j:i\neq j\}\cup\{\pm 2e_i\}, \qquad \Lambda=\Z^n. 
\]
The issue occurs when $\alpha=\pm 2e_i$, since then $\alpha(\Lambda)\subset 2\Z$.

Suppose first that $G$ is simply connected and simple. Then, by checking the classification of compact simple Lie groups and the corresponding root data, one finds that the situation described above is the only case in which the lemma fails.

In general, we can write $G=G'/\Gamma$, where $G'$ is a product of copies of $\operatorname{U}(1)$ and simply connected simple groups, and $\Gamma$ is a finite subgroup of $Z(G')$. Since $\Lambda_{G'}\subset \Lambda_G$, the lemma can fail only when $G'$ contains an $\operatorname{Sp}(n)$ factor and $\alpha=\pm 2e_i$ in the notation above.

Consider such an $\operatorname{Sp}(n)$ factor. Since $G$ has no direct factor isomorphic to $\operatorname{Sp}(n)$, the image of $\Gamma$ under the projection to this factor must be nontrivial, and hence must be all of
\[
Z(\operatorname{Sp}(n))=\Z_2.
\]
Since $\Lambda_{G'}\to \Lambda_{\operatorname{Sp}(n)}$ is surjective and $\Lambda_{\operatorname{PSp}(n)}/\Lambda_{\operatorname{Sp}(n)}=Z(\operatorname{Sp}(n))$, the homomorphism
\begin{equation}\label{eq:map_cocharacter_lattice}
\Lambda_G\to \Lambda_{\operatorname{PSp}(n)}
\end{equation}
is also surjective. Hence, we complete the proof by choosing $\lambda$ whose image under \eqref{eq:map_cocharacter_lattice} is equal to the positive or negative fundamental coweight according to the sign of $\alpha$.
\end{proof}

\begin{proof}[Proof of \Cref{thm:Coulombbranch=Whilb}]
We first assume that $G$ has no direct factor isomorphic to $\operatorname{Sp}(n)$. We need to show that the induced homomorphism
\[
\Psi:\C[\WH(C_T)]\longrightarrow H^G_\bullet(\Omega G)
\]
is an isomorphism. 

Let $\mathcal{Z}\subset C_T\times \WH(C_T)$ be the universal family. Write $C_T=\spec A$. We have $\mathcal{Z}=\spec \NHcl$ and $\WH(C_T)=\spec (\NHcl)^W$ by \Cref{thm:WHilb=nilHecke_closure}. We first show that $\Psi$ is injective. By \Cref{lem:smoothness_WHilb}, $(\NHcl)^W$ is flat over $\C[\lt_\C]^W$, and hence
\[
\NHcl=\C[\lt_\C]\otimes_{\C[\lt_\C]^W}\NHcl^W
\]
is flat over $\C[\lt_\C]$. Let $B$ be the algebra obtained from $\C[\lt_\C]$ by localizing at all the roots. Then we have
\begin{equation}\label{eq:NHcl_injective}
    \NHcl\subset B\otimes_{\C[\lt_\C]}\NHcl.
\end{equation}
On the other hand, we have
\[
B\otimes_{\C[\lt_\C]}\NHcl=\operatorname{cl}_{\NH}(B\otimes_{\C[\lt_\C]} A)=B\otimes_{\C[\lt_\C]} A.
\]
Here, the first equality follows from the universal property of nil-Hecke closure, and the second equality holds because $B$, and hence $B\otimes_{\C[\lt_\C]}A$, is an $\NH$-algebra. Combining with \eqref{eq:NHcl_injective}, we see that the composition
\[
\NHcl\to B\otimes_{\C[\lt_\C]}A\subset \C(C_T)
\]
is injective, and hence $\Psi$ is also injective.

We will now show the surjectivity of $\Psi$. By \cite[Theorem~5.26]{BFN}, it suffices to show that the corresponding homomorphism
\[
\widetilde{\Psi}_t:\NHcl_t\to H^T_\bullet(\Omega G)_t
\]
is surjective, where the subscript means the localization at a point $t\in \lt_\C$ which lies in at most one root hyperplane.

By localization and \cite[Lemma~5.1]{BFN}, we have
\[
H^T_\bullet(\Omega G)_t=H^T_\bullet(\Omega Z_G(t))_t.
\]
If $t$ does not lie in any root hyperplane, then $Z_G(t)=T$, and the surjectivity of $\widetilde{\Psi}_t$ is clear.

Now suppose $t$ lies in exactly one hyperplane $\{\alpha=0\}$, and let $\lambda$ be the cocharacter of $T$ asserted by \Cref{Liegroup_lemma} such that $\langle \alpha,\lambda\rangle=1$. We may prove the surjectivity after replacing $G$ with $Z_G(t)$. In other words, we may assume the semisimple rank of $G$ is $1$. In this case, $H_\bullet^T(\Omega G)$ is generated over
$H_\bullet^T(\Omega T)$ by the Schubert class $[X_\lambda]$
corresponding to $\lambda$; see, for example, the proof of
\cite[Theorem~6.8]{BFN}. Note that $[X_{\lambda}]$ is equal to $D_{\alpha}(t^\lambda)$, which lies in $\NHcl$. This proves the surjectivity. We conclude that $\Psi$ is an isomorphism if $G$ has no direct $\operatorname{Sp}(n)$ factors.

In general, choose a short exact sequence
\[
1\to G\to G'\to Q\to 1
\]
as in \eqref{ses_compact} so that $G'$ does not contain any direct $\operatorname{Sp}(n)$ factor. This is possible because we can embed each direct $\operatorname{Sp}(n)$ factor into
\[
\operatorname{Sp}(n)\operatorname{U}(1)=(\operatorname{Sp}(n)\times \operatorname{U}(1))/\Z_2.
\]
We have $\WHH(C_{T'})\cong C_{G'}$. It follows from \cite[Proposition~3.18]{BFN} that the morphism
\begin{equation}\label{eq:reduction_map}
    \WHH(C_{T'})/\!/\check Q_\C\to \WH(C_{T}) 
\end{equation}
induces an isomorphism
\[
\WHH(C_{T'})/\!/\check Q_\C\to C_G.
\]
It remains to show that \eqref{eq:reduction_map} is an open embedding (even for $G'$ with direct $\operatorname{Sp}(n)$ factors). Let $Z'\in {\widetilde{\bbmu}}^{-1}(0)=\WH(\check T'_\C\times \lt_\C)$, and $Z\in \WH(C_T)$ be its image. We have
\begin{align*}
    T_{Z'}\WH(\check T'_\C\times \lt_\C)&=H^0(Z',N_{Z'/\check T'_\C\times \lt})^w\\
    T_{Z}\WH(C_T)&=H^0(Z,N_{Z/C_T})^w.
\end{align*}
It follows from the exact sequence
\[
0\to \check {\mathfrak{q}}_\C\to H^0(Z',N_{Z'/\check T'_\C\times \lt})^w\to H^0(Z,N_{Z/C_T})^W\to 0.
\]
induced by the $\check Q_\C$-action that \eqref{eq:reduction_map} is \'etale. And since \eqref{eq:reduction_map} is birational and $\WH(C_T)$ is smooth by \Cref{lem:smoothness_WHilb}, it is an open embedding by Zariski's Main theorem.
\end{proof}

\begin{cor}\label{cor:image_of_jmath}
Let
\[
U_T=\bigcup_{\alpha\in \Phi} \{(t,a)\in C_T: \alpha(a)=0,\ \alpha^\vee(t)\neq 1\},
\]
and let $U'_T\subset \WH(C_T)$ be the preimage of $U_T/W$ under the morphism $\WH(C_T)\to C_T/W$. Then
\[
C_G=\WH(C_T)\setminus U'_T.
\]
\end{cor}

\begin{proof}
We have seen in the proof of \Cref{lem:smoothness_WHilb} that if $Z\in \WH(C_T)$ contains a point $(t,a)\in C_T$, then $\operatorname{Stab}_W(t)\supset \operatorname{Stab}_W(a)$. In particular, if $\alpha(a)=0$, then $s_\alpha(t)=t$.

First suppose that $G$ does not contain any direct $\operatorname{Sp}(n)$ factor. We claim that $s_\alpha(t)=t$ implies $\alpha^\vee(t)=1$, which proves $U'_T=\varnothing$. Indeed, by \Cref{Liegroup_lemma}, we can choose a cocharacter $\lambda$ of $T$ such that $\langle \alpha,\lambda\rangle=1$. Regarding $\lambda$ as a character of $\check T_\C$, if $s_\alpha(t)=t$, then
\[
\lambda(t)=\lambda(s_\alpha(t))=s_\alpha(\lambda)(t)=\alpha^\vee(t)^{-1}\lambda(t),
\]
and hence $\alpha^\vee(t)=1$. 

In general, choose a short exact sequence
\[
1\to G\to G'\to Q\to 1
\]
as in \eqref{ses_compact} so that $G'$ does not contain any direct $\operatorname{Sp}(n)$ factor. The result for $G'$ implies that $U'_{T}$ is disjoint from the image of $\WH(C_{T'})//\check Q_\C\to \WH(C_T)$.

Conversely, suppose that $Z\in \WH(C_T)$ has support disjoint from $U_T$. We can write
\[
Z=\bigsqcup_{\sigma\in W/H}\sigma(Z_x),
\]
where $x=(t,a)\in C_T$, $H=\operatorname{Stab}_W(x)$, and $Z_x\in H\text{-}\operatorname{Hilb}_{\lt_\C}(C_T)$ is supported at $x$. Let $t'\in \check T'_\C$ be a lift of $t$. We now show that
\[
\operatorname{Stab}_W(t')\supset H.
\]
Recall $H=\operatorname{Stab}_W(a)$ is generated by reflections. Let $s_\alpha\in H$. Since $s_\alpha$ fixes $x=(t,a)$, it fixes $a$, and hence $\alpha(a)=0$. Since the support of $Z$ is disjoint from $U_T$, we have $\alpha^\vee(t)=1$. Therefore $\alpha^\vee(t')=1$, and hence $s_\alpha(t')=t'$. Thus every reflection generator of $H$ fixes $t'$, so $\operatorname{Stab}_W(t')\supset H$.

Let $x'=(t',a)$. Then $\operatorname{Stab}_W(x)=\operatorname{Stab}_W(x')$, and it suffices to lift $Z_x$ to a transverse $H$-subscheme $Z_{x'}$ of $C_{T'}$ supported at $x'\in\bbmu^{-1}(0)$, because $Z_{x'}$ then extends to a transverse $W$-subscheme of $\bbmu^{-1}(0)$ lifting $Z$.

Note that there is an isomorphism
\[
\hat\oh_{x'}\cong\hat\oh_x[[\check{\mathfrak q}_\C]].
\]
If $Z_x$ is defined by an ideal $I\subset\oh_x$, then, under the above isomorphism, we may define $Z_{x'}$ by the ideal $I'\subset\oh_{x'}$ generated by $I$ and $\check{\mathfrak q}_\C$. This completes the proof.
\end{proof}

\section{Nonabelian lifting}\label{section:nonabelian_lifting}
As mentioned in the introduction, 3d mirror symmetry predicts that for each Hamiltonian $G$-manifold $Y$, there is a mirror complex Lagrangian subvariety $\bbL^G_Y$ in $C_G$. We refer to \cite{Paper1} for further discussion of this subject.

\subsection*{Abelian case}
We now review the construction of $\bbL^T_Y$ as proposed by Teleman \cite{tel2014}. Recall $C_T$ is just $T^*\check T_\C$. 

Suppose $Y$ is a symplectic manifold and $(Y^\vee,f)$ is its 2d mirror Landau--Ginzburg model. Teleman conjectured that a Hamiltonian $T$-action on $Y$ is mirror to a holomorphic map $\fz=\fz_T:Y^\vee\to \check T_\C$. In the next section, we will define $\fz$ on the set of Lagrangian branes with respect to a $B$-field. In this section, we will simply assume $\fz$ exists.

Let $\Gamma_{\fz_T}\subset Y^\vee\times \check T_\C$ be the graph of $\fz_T$ and
\[
N^*_{\Gamma_{\fz_T}}\subset T^*(Y^\vee\times \check T_\C)\cong T^*Y^\vee\times T^*\check T_\C
\]
be its conormal. Let $\Gamma_{df}\subset T^*Y^\vee$ be the graph of $df$. We define 
\[
\bbL^T_Y=N^*_{\Gamma_{\fz_T}}\circ \Gamma_{df}\subset T^*\check T_\C
\]
to be the Lagrangian composition. Explicitly, we can identify $C_T$ with $\check T_\C\times \lt_\C$ and set
\begin{equation}\label{eq:ZTY}
    Z^T_Y=\{(p,z,a)\in Y^\vee\times \check T_\C\times\lt_\C:z=\fz_T(p),df|_p=\fz_T^* da|_p\}.
\end{equation}
Here $da$ is regarded as a left-invariant differential form on $\check T_\C$ whose value at the identity is $a$. Then $\bbL^T_Y$ is the image of $Z^T_Y$ under the natural projection $\check Y\times \check T_\C\times\lt_\C\to \check T_\C\times\lt_\C$. 

If $\bbL_Y^T$ is reduced, then its smooth locus is Lagrangian, and we will refer to $\bbL_Y^T$ as Lagrangian by abuse of notation. In general, the Lagrangian composition above should be understood in the derived sense; it then defines a $0$-shifted Lagrangian in $T^*\check T_\C$; see \cite{PTVV}. For simplicity, we henceforth assume that $\bbL^T_Y$ is reduced. Replacing $\bbL^T_Y$ by its closure if necessary, we may assume that $\bbL^T_Y$ is a closed subscheme of $C_T$.

\begin{exmp}\label{exam1}
 	Let $Y=\mathbb{P}^1$ with the standard toric $T=\operatorname{U}(1)$-action. The Hori-Vafa mirror of $Y$ is $(Y^\vee=\Cx, f=z+1/z)$, and the Teleman map $\fz_T:\Cx\to \Cx=\check T_\C$ is the identity map. In this case, $\bbL^T_Y$ is simply the graph of $df$. If $z$ and $a$ are the base and fiber coordinates of $T^*\check T_\C=T^*\Cx$, then $\bbL^T_Y$ is defined by the equation
 	\[a=z-\frac{1}{z}.\]
 \end{exmp}
 
  \begin{exmp}\label{exam2}
 	We still consider $Y=\mathbb{P}^1$, but assume $T=S^1$ is the double cover of the $S^1$ in \Cref{exam1}. In other words, $\{\pm 1\}\subset S^1$ lies in the kernel of the action. We still have $(Y^\vee,f)=(\Cx,w+1/w)$, but $\fz:Y^\vee\to \check T_\C=\Cx$ is the square map $w\mapsto z=w^2$. $\bbL^T_Y$ is defined by the equation
 	\[4a^2=z+\frac{1}{z}-2.\]
 \end{exmp}

Now suppose that the $T$-action on $Y$ extends to a Hamiltonian $G$-action. We propose that $\bbL^G_Y\subset C_G$ can be obtained from $\bbL^T_Y\subset C_T$ by applying the transverse $W$-Hilbert scheme functor. More precisely, we conjecture that
\begin{conjj}\label{conj}
Suppose that the $T$-action on $Y$ extends to a Hamiltonian $G$-action. Then:
\begin{enumerate}[(i)]
\item $\bbL^T_Y\subset C_T$ is $W$-invariant;
\item the natural morphism $\WH(\bbL^T_Y)\to \bbL^T_Y/W$ is an isomorphism;
\item the morphism $\WH(\bbL^T_Y)\to \WH(C_T)$ factors through the image of $C_G\to \WH(C_T)$.
\end{enumerate}
\end{conjj}
We refer the reader to \Cref{section:introduction} for some comments on \Cref{conj}. 
\begin{rem}
The diagram \eqref{Couldiag1} induces a Lagrangian correspondence
\[
\spec H^T_\bullet(\Omega G)
\subset
\spec H^T_\bullet(\Omega T)\times \spec H^G_\bullet(\Omega G).
\]
If \Cref{conj} holds, then
\[
\bbL^G_Y=\bbL^T_Y/W\subset C_G
\]
can be understood as (union of components of) the Lagrangian composition
\[
\spec H^T_\bullet(\Omega G)\circ \bbL^T_Y.
\]
Consequently, $\bbL^G_Y$ is Lagrangian in $C_G$ in either of the following senses: classically, provided that $\bbL^G_Y$ is reduced; or derivedly, by interpreting the above composition as a derived Lagrangian composition in the sense of shifted symplectic geometry \cite{PTVV}. The purpose of \Cref{conj} is to ensure that no information is lost
under this Lagrangian correspondence.
\end{rem}
 
\begin{exmp}
The $S^1$-action on $Y=\mathbb{P}^1$ in \Cref{exam1} naturally extends to a $G=\operatorname{SO}(3)$-action. We have
\[
\C[\bbL^T_Y]
=
\C[a,z^{\pm 1}]/(z-z^{-1}-a)
=
\C[z+z^{-1}]\oplus \C[z+z^{-1}]a.
\]
Hence $\C[\bbL^T_Y]$ is an $\NH$-algebra by \Cref{GinLemma}, and therefore $\WH(\bbL^T_Y)=\bbL^T_Y/W$. Note that $C_G=\WH(C_T)$ in this case.
\end{exmp}

\begin{exmp}
The $S^1$-action on $Y=\mathbb{P}^1$ in \Cref{exam2} naturally extends to a $G=\operatorname{SU}(2)$-action. Similarly, one can check that $\C[\bbL^T_Y]$ is an $\NH$-algebra. In this case, $\C[C_G]$ is obtained from $\C[\WH(C_T)]$ by adjoining
\[
g=\frac{z+z^{-1}-2}{a^2}.
\]
The restriction of $g$ to $\bbL^T_Y=\{z+z^{-1}-2=4a^2\}$ is regular, and is in fact equal to the constant function $4$. Therefore, $\WH(\bbL^T_Y)\subset C_G$.
\end{exmp}

\begin{lem}\label{lem:criterion_bbLG1}
Let $\bbL\subset C_T$ be a $W$-invariant closed reduced subscheme.
\begin{enumerate}
\item The natural morphism $\WH(\bbL)\to \bbL/W$ is an isomorphism if and only if $\C[\bbL]$ is an $\NH$-algebra.
\item Suppose furthermore that, for any short exact sequence
\[
1\to G\to G'\to Q\to 1
\]
as in \eqref{ses_compact}, there exists a $W$-invariant closed subscheme $\bbL'\subset C_{T'}$ such that $\WHH(\bbL')=\bbL'/W$ and $\bbL$ is equal to the image of 
\[
\bbL'\cap \bbmu^{-1}(0)\to \bbmu^{-1}(0)/\check Q_\C\cong C_T,
\]
then $\WH(\bbL)\subset C_G$.
\end{enumerate}
\end{lem}
\begin{proof}
    (1): Assume the natural morphism $\WH(\bbL)\to \bbL/W$ is an isomorphism, and let $\mathcal{Z}\subset \bbL/W\times \bbL$ be the universal family. By \eqref{eq:Z_in_fibre_product}, we have
    \[
    \mathcal{Z}\subset \bbL\times_{\bbL/W}\bbL/W=\bbL.
    \]
    On the other hand, note that both $\mathcal{Z}\to \bbL/W$ and $\bbL\to\bbL/W$ are finite and surjective. The degree of the former is $\dim R=|W|$, while the degree of the latter is at most $|W|$. Since $\C[\bbL]$ is reduced and $\C[\bbL]\to \C[Z]$ is surjective, the only possibility is $\mathcal{Z}=\bbL$, and in particular $\C[\bbL]$ is an $\NH$-algebra. The converse is obvious.

    (2): The assumption implies that $\WH(\bbL)$ is equal to the image of 
    \[
    \WHH(\bbL')\cap {\widetilde\bbmu}^{-1}(0)\to {\widetilde\bbmu}^{-1}(0)/\check Q_\C\to \WH(C_T).
    \]
    The result then follows from \Cref{thm:Coulombbranch=Whilb}.
\end{proof}

\begin{exmp}
Let $X_G$ be the full flag variety of $G_\C$. \citeauthor{MR2397456} constructed a mirror $(X^\vee_G,W_G)$ of $X_G$ in \cite{MR2397456}. We refer the reader to loc. cit. for more details. The following properties can be checked from the construction:
\begin{enumerate}
\item The results of \cite{MR2397456} show that $\C[\bbL^T_{X_G}]$ is isomorphic to the equivariant quantum cohomology of $X_G$\footnote{The author of \cite{MR2397456} assumes that $G$ is simply connected. However, the construction also works for any connected compact group $G$, and the quantum-cohomological statements remain true if one uses the equivariant Novikov parameters introduced in \cite{chan2025quantumcohomologyshiftoperators}.}. It follows from \Cref{exmp:NH_action_on_HG} that this ring is an $\NH$-algebra.
\item Given a short exact sequence as in~\eqref{ses_compact}, there is a $\check Q_\C$-action on $X_{G'}^\vee$ such that $X_{G'}^\vee/\check Q_\C\cong X_ G^\vee$, and $W_{G'}$ is the pullback of $W_G$ via the quotient map. One can check that $\bbL^{T'}_{X_{G'}}$ satisfies condition (2) of \Cref{lem:criterion_bbLG1}.
\end{enumerate}
\end{exmp}
For each root $\alpha$ of $(G,T)$, it is easy to see that
\[
\bbL_Y^{T^{s_\alpha}}=\bbL^T_Y\cap V(\alpha),
\]
where $V(\alpha)\subset T^*\check T_\C$ is the closed subscheme defined by the function $\alpha$.

\begin{lem}\label{lem:criterion_bbLG2}
Let $\bbL\subset C_T$ be a $W$-invariant closed Lagrangian subvariety. Suppose
\[
\bbL\cap V(\alpha_i)
\]
is reduced for every $i$. Then:
\begin{enumerate}[(i)]
    \item The natural morphism $\WH(\bbL)\to \bbL/W$ is an isomorphism if and only if, for every root $\alpha$ and every $(t,a)\in\bbL$ satisfying $\alpha(a)=0$, one has $ s_\alpha(t)=t$.

    \item One has $\WH(\bbL)\xrightarrow{\cong}\bbL/W\subset C_G$ if and only if, for every root $\alpha$ and every $(t,a)\in\bbL$ satisfying $\alpha(a)=0$, one has $\alpha^\vee(t)=1$.
\end{enumerate}
\end{lem}
\begin{rem}\label{rem:remove_reducedness_assumption}
Without assuming that $\bbL\cap V(\alpha_i)$ is reduced, \Cref{lem:criterion_bbLG2} remains valid if the pointwise conditions in \textup{(i)} and \textup{(ii)} are replaced, respectively, by the scheme-theoretic inclusions
\[
\bbL\cap V(\alpha_i)\subset \bbL^{s_i}
\qquad\text{and}\qquad
\bbL\cap V(\alpha_i)\subset V(\alpha_i^\vee-1),
\]
where $\bbL^{s_i}$ denotes the scheme-theoretic fixed-point subscheme of $s_i$; see \cite{FixedPointSchemes}.
\end{rem}

\begin{proof}[Proof of \Cref{lem:criterion_bbLG2}]
(i): First suppose that $\C[\bbL]$ is an $\NH$-algebra, and let $(t,a)\in \bbL$ satisfy $\alpha_i(a)=0$. Let $z^\lambda:\check T_\C\to \Cx$ be a character. Then
\[
z^\lambda(t)-z^\lambda(s_i(t))
=
\alpha_i(a)D_i(z^\lambda)(t,a)
=
0.
\]
Since characters separate points of $\check T_\C$, it follows that $s_i(t)=t$.

Conversely, suppose that the conclusion of (i) is satisfied. We want to construct an action of $D_i$ on $\C[\bbL]$ for each $i$. For this purpose, we may replace $G$ by the subgroup whose
complexification is generated by $T_\C$ and the root subgroups
corresponding to $\pm\alpha_i$, and hence assume that $G$ has
semisimple rank one.

By \Cref{GinLemma}, it suffices to show that
\[
\varphi:\C[\bbL]^{s_{i}}\oplus \C[\bbL]^{s_{i}}\xrightarrow{(1,\alpha_i)} \C[\bbL]
\]
is an isomorphism. Since $s_i(\alpha_i)=-\alpha_i$, we have
\[
\C[\bbL]^{s_{i}}\cap\C[\bbL]^{s_{i}}\alpha_i=0.
\]
To show the injectivity of $\varphi$, it suffices to check that $\alpha_i$ is not a zero divisor. Since $\bbL$ is reduced, this is equivalent to saying that $\alpha_i$ is not identically zero on any irreducible component of $\bbL$. Suppose, for contradiction, that $\alpha_i$ vanishes identically on an irreducible component $Z\subset \bbL$. Then the assumption implies that $Z$ is contained in the fixed locus
\[
(\check T_\C\times \lt_\C)^{s_i}=\check T_\C^{s_i}\times \lt_\C^{s_i}.
\]
This fixed locus is a symplectic subvariety of codimension $2$ in $C_T$. Hence no irreducible component of dimension $\dim \bbL$ contained in it can be a Lagrangian in $C_T$, a contradiction.

It remains to show that $g\in \operatorname{Im}(\varphi)$ for any $g\in \C[\bbL]$. Since $\operatorname{Im}(\varphi)$ is an algebra, we only need to check this when $g\in \lt^*$ or when $g=z^\lambda$ for some cocharacter $\lambda$ of $T$. The first case is clear since
\[
\lt^*=(\lt^*)^{s_i}\oplus \C\alpha_i.
\]
For the second case, since
\[
z^\lambda=\frac{1}{2}(z^\lambda+z^{s_i\lambda})+\frac{1}{2}(z^\lambda-z^{s_i\lambda}),
\]
it suffices to show that $z^\lambda-z^{s_i\lambda}\in \operatorname{Im}(\varphi)$. Finally, since $\bbL\cap V(\alpha_i)$ is reduced, the assumption implies that $z^\lambda-z^{s_i\lambda}$ is divisible by $\alpha_i$ in $\C[\bbL]$, and hence $z^\lambda-z^{s_i\lambda}\in \operatorname{Im}(\varphi)$ as desired.

(ii): It is a direct consequence of \Cref{cor:image_of_jmath}.
\end{proof}
Recall the definition of $Z^T_Y$ in \eqref{eq:ZTY}. By \Cref{lem:criterion_bbLG2}, the condition that $\WH(\bbL^T_Y)=\bbL^T_Y/W\subset C_G$ is equivalent to the following conjecture.

\begin{conjj}\label{conj2}
For every root $\alpha$ and every $(p,z,a)\in Z^T_Y$ satisfying $\alpha(a)=0$, one has $\alpha^\vee(\fz(p))=\alpha^\vee(z)=1$.
\end{conjj}
In the remainder of the paper, we formulate a more precise version of \Cref{conj2} using Lagrangian Floer theory and prove it.
\section{The Teleman map and superpotential with an equivariant \texorpdfstring{$B$}{B}-field}\label{section:Teleman_and_f}
We let $Y$ be a Hamiltonian $G$-space. We write $\omega$ and $\mu$ for the symplectic form and moment map on $Y$.
\begin{df}
    An equivariant $B$-field on $Y$ is a $G$-equivariantly closed
    two-form $\tilde{B}=B-\kappa$. In other words, $B$ is a $G$-invariant closed two-form, and $\kappa:Y\to \lig^*$ is a $G$-equivariant map so that $\iota_{\xi^\sharp}B=d\langle \kappa,\xi\rangle$ for all $\xi\in \lig$.
\end{df}
If $H$ is a subgroup of $G$, we write $\mu_H$ and $\kappa_H$ for the restrictions of $\mu$ and $\kappa$ to $\operatorname{Lie}(H)^*$.

\begin{df}
    A $T$-invariant Lagrangian brane on $Y$ is a pair $\bL=(L,E)$, where $L\subset Y$ is a compact connected $T$-invariant Lagrangian submanifold, and $E$ is a $T$-equivariant unitary line bundle on $L$ equipped with a $T$-invariant unitary connection $\nabla_E$ having curvature $F_E=iB|_L$.
\end{df}

Intuitively, the Teleman map is defined as a combination of the moment map and monodromy. However, we have to be careful about what we mean by monodromy when the bundles under consideration are not flat.

Let $\bL=(L,E)$ be a $T$-invariant Lagrangian brane on $Y$, and let $p\in L$. Consider the orbit map
\begin{equation}\label{eq:orbit_map}
    o=o_p:T\to L,\qquad t\mapsto tp.
\end{equation}
The pullback line bundle $o_p^*(E)$ is flat: indeed, its curvature is exact and $T$-invariant, and hence must vanish. Its monodromy determines an element
\[
\hol^p_T(E)\in \Hom(\pi_1(T),\operatorname{U}(1))\cong \check T.
\]
Note that the Lie algebra of $\check T$ is canonically identified with $\Hom(\pi_1(T),i\R)\cong i\lt^*$. We will also identify $\Lambda$ with $\pi_1(T)$ and with the kernel of the exponential map $\exp:\lt\to T$. If $\lambda:S^1\to T$ is a cocharacter, we write
\[
\hol^p_\lambda(E)=\hol^p_T(E)(\lambda).
\]
\begin{lem}\label{lem:indep_of_p}
The value $\exp(-i\kappa_T(p))\hol^p_T(E)\in \check T$ is independent of the point $p\in L$. Moreover, if $\bL$ is $T$-equivariantly Hamiltonian isotopic to $\bL'=(L',E')$, and $p'\in L'$, then $\exp(-i\kappa_T(p))\hol^p_T(E)=\exp(-i\kappa_T(p'))\hol^{p'}_T(E')$.
\end{lem}
\begin{proof}
Let $p,q\in L$, and choose a smooth path
$c:[0,1]\to L$ from $p$ to $q$. Let
$\lambda:S^1\to T$ be a cocharacter, and choose $\xi\in\mathfrak t$
such that
\[
\exp(\theta\xi)=\lambda(e^{2\pi i\theta})
\]
for every $\theta\in\mathbb R$. Define
\[
\Psi:S^1\times[0,1]\longrightarrow L,
\qquad
\Psi(z,s)=\lambda(z)\cdot c(s).
\]
Then
\begin{align*}
    \hol^p_\lambda(E)(\hol^q_\lambda(E))^{-1}
    &=\exp\left(-\int_{S^1\times [0,1]} i \Psi^*B\right)\\
    &=\exp\left(-\int_0^1 i\langle d(\kappa_T\circ c),\xi\rangle\right)\\
    &=\exp(-i\langle \kappa_T(q),\xi\rangle+i\langle \kappa_T(p),\xi\rangle)\\
    &=\exp(-i\kappa_T(q))(\lambda)\left(\exp(-i\kappa_T(p))(\lambda)\right)^{-1}.
\end{align*}
    Therefore, $\exp(-i\kappa_T(p))\hol^p_T(E)=\exp(-i\kappa_T(q))\hol^q_T(E)$ as claimed. Since this equality holds after evaluation on every cocharacter of
    $T$, we conclude that
    \[
    \exp\bigl(-i\kappa_T(p)\bigr)\operatorname{Hol}_T^p(E)
    =
    \exp\bigl(-i\kappa_T(q)\bigr)\operatorname{Hol}_T^q(E)
    \]
as elements of $\check T$. The second part of the lemma can be proved by the same method.
\end{proof}
\begin{df}
    Let $\bL=(L,E)$ be a $T$-invariant Lagrangian brane on $Y$, and let $p\in L$. We define
    \[
    \HH_T(E)= \exp(-i\kappa_T(p))\hol^p_T(E)\in \check T,
    \]
    and
    \[
     \fz_T(\bL)=\exp(-\mu_T(L))\cdot\HH_T(E)\in \check T_\C\cong \exp(\lt^*)\cdot\check T.
    \]
    Here, $\mu_T|_L$ is constant because $L$ is $T$-invariant and connected, and we denote its value by $\mu_T(L)$. We call $\HH_T(E)$ the equivariant monodromy of $E$ and $\fz_T$ the Teleman map.
\end{df}
By \Cref{lem:indep_of_p}, $\HH_T(E)$ and $\fz_T(\bL)$ are independent of the point $p\in L$. If $\alpha:S^1\to T$ is a cocharacter, we write $\HH_\alpha(E)=\HH_T(E)(\alpha)$. 
\begin{rem}
Suppose that $L$ is a torus fiber of an SYZ fibration. Then $\bL$ can be interpreted as a point of the 2d mirror $Y^\vee$. We have
\[
T_{\bL}Y^\vee
\cong H^1(L;\R)\oplus iH^1(L;\R)
\cong H^1(L;\C).
\]
Choose a Riemannian metric on $L$ and identify $H^1(L;\R)$ with the space of harmonic forms. For $a,b\in H^1(L;\R)$ and $\xi\in\lt$, one has
\[
d\log\fz_T(a+ib)(\xi)
=
i\,a(\xi^\sharp)-b(\xi^\sharp),
\]
see, for example, \cite[Proposition~2.10]{Paper1}. Thus $d\log\fz_T$ is complex linear, and consequently $\tau_T$ is
holomorphic.
\end{rem}

Next we discuss the superpotential. Fix a $T$-invariant Lagrangian brane $\bL=(L,E)$ on $Y$. We assume that $L$ is equipped with a $T$-invariant relative spin
structure; see \cite{FOOOII,KLZ}. We work over the universal Novikov field
\begin{equation}\label{NovField}
    \F=
    \left\{
    \sum_{i=0}^\infty a_i\T^{\lambda_i}
    :
    a_i\in\C,\ 
    \lambda_i\in \R,\ 
    \lambda_i<\lambda_{i+1},\ 
    \lim_{i\to\infty}\lambda_i=\infty
    \right\}.
\end{equation}
We write
\begin{align}\label{m0}
    f(\bL)=f_Y(\bL)
    &=\sum_{\beta}n_\beta \exp\left(i\int_\beta B\right)\hol_{\partial \beta}(E)\T^{\int_\beta\omega}\in \F,\\
    df(\bL)=df_Y(\bL)
    &=\sum_{\beta}n_\beta \exp\left(i\int_\beta B\right) \hol_{\partial \beta}(E)\T^{\int_\beta\omega}[\partial \beta]\in H_1(L;\F).
\end{align}
where the sums run over the Maslov-index-two disc classes $\beta$ in $\pi_2(Y,L)$, $\T$ is the Novikov parameter, and $n_\beta$ is the virtual count of holomorphic discs representing
the class $\beta$ (see \cite{AurouxComp,FOOOI}). We remark that the expression
\[
\exp\left(i\int_\beta B\right)\hol_{\partial\beta}(E)
\]
depends only on the relative homotopy class $\beta\in\pi_2(Y,L)$
and not on the choice of representative. This follows from an
argument similar to that used in Lemma~6.3; see
\cite[Lemma~2.7]{FukB}.

\section{Main theorems}\label{section:main_theorems}
For the remainder of this paper, we fix the following data.
\begin{itemize}
    \item $Y$ is a symplectic manifold equipped with an $\omega$-compatible almost complex structure $J$ and a Hamiltonian $G$-action preserving $J$. We denote the moment map by
    \[
        \mu_G:Y\to\lig^*.
    \]
    If $Y$ is noncompact, we assume that it is either convex at infinity or geometrically bounded.
    \item $\tilde B=B-\kappa$ is a $G$-equivariant B-field on $Y$.
    
    \item $\bL=(L,E)$ is a $T$-invariant Lagrangian brane on
    $(Y,\widetilde B)$, and $L$ is equipped with a $T$-invariant
    relative spin structure; see \cite{FOOOII,KLZ}. We assume $E$ is topologically trivial.
\end{itemize}
We review the relevant aspects of Lagrangian Floer theory in \Cref{appendix:Floer_theory}. \Cref{thm1,thm2} below are the main theorems in the second half of this paper; their proofs will be given in the next section. We denote by $Z(H)$ the center of a group $H$.
\begin{thm}\label{thm1}
If $HF^\bullet(\bL,\bL)\neq 0$, then $\fz_T(\bL)\in Z(\check G_\C)$.
\end{thm}
Here, $\check G_\C$ is the Langlands dual group of $G_\C$. It is characterized, up to isomorphism, by the condition that $\check T_\C$ is a maximal torus of $\check G_\C$ and that the root datum of $\check G_\C$ is dual to the root datum of $G_\C$. In particular, if $\alpha\in \Phi=\Phi(G_\C,T_\C)$, then the corresponding coroot $\alpha^\vee$ is a root of $\check G_\C$ with respect to $\check T_\C$. Note that
\begin{equation}\label{eq:Z(G)=int_ker}
    Z(\check G_\C)=\bigcap_{\alpha\in \Phi}\ker(\alpha^\vee)\subset \check T_\C.
\end{equation}

If the minimal Maslov number of $L$ is at least $2$, and $H^\bullet(L)$ is generated in degree 1, then \Cref{H1gen} in \Cref{appendix:Floer_theory} implies that $df(\bL)=0$ implies that $HF^\bullet(\bL,\bL)=H^\bullet(L;\F)\neq 0$. We thus obtain the following corollary of \Cref{thm1}. 
\begin{cor}\label{cor1} Assume the minimal Maslov number of $L$ is at least $2$, and $H^\bullet(L)$ is generated in degree 1. If $df(\bL)=0$, then $\fz_T(\bL)\in Z(\check G _\C)$.
\end{cor}
There is a natural homomorphism $o_*:H_1(T;\F)\to H_1(L;\F)$ induced by the orbit map \eqref{eq:orbit_map}. An element $\alpha\in\Phi$ induces a linear map $H_1(T;\F)\to\F$, which we also denote by $\alpha$. The following theorem refines \Cref{cor1} in accordance with~\eqref{eq:Z(G)=int_ker}.
\begin{thm}\label{thm2}
Assume the minimal Maslov number of $L$ is at least $2$, and $H^\bullet(L)$ is generated in degree 1. Let
$\alpha\in\Phi(G,T)$. If $df(\bL)\in o_*(\ker(\alpha))$, then $\fz_T(\bL)\in \ker (\alpha^\vee)$.
\end{thm}
Let $a\in H_1(T;\F)$. Then the equation $df(\bL)=o_*(a)$ is analogous to the condition $df|_p=\fz_T^*da|_p$ in \eqref{eq:ZTY}. Thus, \Cref{thm2} gives a Floer-theoretic reformulation of \Cref{conj2}.

Although the proofs of \Cref{thm1,thm2} are logically independent of the results of the earlier sections, the idea of the proofs is to reduce these theorems to statements parallel to \Cref{lem:criterion_bbLG1,lem:criterion_bbLG2}. The argument below also follows the flow of the proof of \Cref{thm:Coulombbranch=Whilb}.

\begin{lem}\label{lem:without_spn_thm12}
    Suppose that $G$ has no direct factor isomorphic to
    $\operatorname{Sp}(n)$. Then, for any $\alpha\in \Phi$
    \[
    \ker(\alpha^\vee)=\check T_\C^{s_\alpha}.
    \]
    In particular,
    \[
    Z(\check G_\C)=\check T_\C^W.
    \]
\end{lem}

\begin{proof}
    It is clear that $\ker(\alpha^\vee)\subset \check T_\C^{s_\alpha}$. We prove the reverse inclusion.

    By \Cref{Liegroup_lemma}, there exists a cocharacter $\lambda$ of $T$ such that $\langle\alpha,\lambda\rangle=1$. We regard $\lambda$ as a character of $\check T_\C$. Let $z\in \check T_\C^{s_\alpha}$. Then
    \[
        \lambda(z)=\lambda(s_\alpha(z))=s_\alpha(\lambda)(z)=\alpha^\vee(z)^{-1}\lambda(z).
    \]
    Hence $\alpha^\vee(z)=1$, or equivalently $z\in \ker(\alpha^\vee)$. This proves the first statement. The second statement follows from the first statement together with \eqref{eq:Z(G)=int_ker}.
\end{proof}

As a result, \Cref{thm1,thm2} are equivalent to the following lemmas when $G$ does not contain a direct $\operatorname{Sp}(n)$ factor.
\begin{lem}\label{lem:thm1}
If $HF^\bullet(\bL,\bL)\neq 0$, then $\fz_T(\bL) \in \check T_\C^W$.
\end{lem}
\begin{lem}\label{lem:thm2}
Assume the minimal Maslov number of $L$ is at least $2$, and $H^\bullet(L)$ is generated in degree 1. Let $\alpha\in\Phi(G,T)$. If $df(\bL)\in o_*(\ker(\alpha))$, then $\fz_T(\bL)\in \check T_\C^{s_\alpha}$.
\end{lem}

We now reduce the general case to the case when $G$ has no $\operatorname{Sp}(n)$ factor.  
Consider a short exact sequence
\begin{align}
    1\to G&\to G'\to Q\to 1\\
    1\to T&\xrightarrow{\rho} T'\to Q\to 1.
\end{align}
as in \eqref{ses_compact} and \eqref{ses_torus}. We can write $G'=(G\times Q')/\Gamma$ for some finite subgroup $\Gamma\subset Z(G\times Q')$. Choose a compact manifold $Z$ equipped with a free $G\times Q'$-action. We will choose $Z$ so that $\pi_1(Z)=\pi_2(Z)=0$, and set
\[
Y'=(Y\times T^*Z)/\Gamma.
\]
The product $Y\times T^*Z$ carries the symplectic form
\[
\omega_Y+\omega_{\mathrm{can}},
\]
where $\omega_{\mathrm{can}}=d\gamma_{\mathrm{can}}$ and $\gamma_{\mathrm{can}}$ is the canonical one-form on $T^*Z$.
The natural $G\times Q'$-action on $Y\times T^*Z$ is Hamiltonian,
with moment map
\[
\mu_{Y\times T^*Z}(y,x)(\xi)
=
\mu_Y(y)(\xi_1)-\gamma_{\mathrm{can}}(\xi^\sharp)(x),\qquad \xi=(\xi_1,\xi_2)\in \lig\oplus \mathfrak q'.
\]
Since $\Gamma$ acts freely and preserves these data, the symplectic
form and moment map descend to a symplectic form $\omega'$ and a
$G'$-moment map $\mu_{G'}'$ on $Y'$.

Similarly, the pullback of $\Tilde{B}$ to $Y\times T^*Z$ descends to an equivariant B-field $\Tilde{B}'=B'-\kappa'$ on $Y'$, and the pullback of $E$ to $L\times Z$ descends to a unitary complex line bundle $E'$ on $L'=(L\times Z)/\Gamma$ with $F_{E'}=iB'|_{L'}$. We set $\bL'=(L',E')$.

The following theorem addresses \Cref{conj}~(3) in view of \Cref{lem:criterion_bbLG1}.
\begin{thm}\label{thm3}
\begin{enumerate}
    \item[(i)] $HF^\bullet(\bL',\bL')\cong HF^\bullet(\bL,\bL)\otimes H^\bullet(Z;\F)$.
    \item[(ii)] $df_{Y'}(\bL')$ is equal to the image of $df_Y(\bL)$ under the natural map
 \[
 H_1(L;\F)\to H_1(L;\F)\oplus H_1(Z;\F)\xrightarrow{\cong} H_1(L\times Z;\F)\to H_1(L';\F),
 \]
 and $\fz_{T}(\bL)$ is equal to the image of $\fz_{T'}(\bL')$ under the quotient map
 \[
 \check \rho:\check T'_\C\to \check T_\C.
 \]
\end{enumerate}

\end{thm}
\begin{prop}\label{prop:reduction_to_no_sp(n)}
\begin{enumerate}
    \item \Cref{thm1} follows from \Cref{lem:thm1} and \Cref{thm3}.
    \item \Cref{thm2} follows from \Cref{lem:thm2} and \Cref{thm3}.
\end{enumerate}
\end{prop}
\begin{proof}  
    We only prove the first statement; the second statement can be proved similarly, and is actually not needed in the remainder of the paper. We can find a short exact sequence \eqref{ses_compact} such that $G'$ does not contain a direct $\operatorname{Sp}(n)$ factor. Since the map $\check G'_\C\to \check G_\C$ sends the center to the center, by \Cref{thm3}, it suffices to show that $\fz_{T'}(\bL')\in Z(\check G'_\C)$. By \Cref{lem:without_spn_thm12}, this is equivalent to showing that $\fz_{T'}(\bL')$ is fixed by the Weyl group, which is the content of \Cref{lem:thm1}.
\end{proof}
\begin{proof}[Proof of \Cref{thm3}]
    We first prove the second statement of part (ii). Let $p\in L$, let $q\in Z$, and let $p'\in L'$ be the image of $(p,q)$. The statement $\check \rho(\fz_{T'}(\bL'))=\fz_{T}(\bL)$ is equivalent to
    \[
    \begin{aligned}
    \rho^*(\mu_{T'}(L'))&=\mu_T(L),\\
    \operatorname{Hol}_{T'}^{p'}(E')|_T
    &=\operatorname{Hol}_T^p(E),\\
    \rho^*(\kappa_{T'}'(p'))&=\kappa_T(p).
    \end{aligned}
    \]
    These identities are clear from the construction.

    Since $\pi_2(Z)=0$, we have an isomorphism
    \[
    i:\pi_2(Y,L)\xrightarrow{\cong} \pi_2(Y',L'),\qquad \beta\mapsto \pr_*(\beta\times 1),
    \]
    where $\pr:Y\times T^*Z\to Y'$ is the projection. Let
    \[
    \bL_Z=(L_Z,E_Z)=\bigl(L\times Z,\operatorname{pr}_L^*E\bigr)
    \]
    be the pullback Lagrangian brane on $Y\times T^*Z$. Since $T^*Z$ is exact, any pseudoholomorphic disc in $Y\times T^*Z$ bounding $L_Z$ must be constant in the $T^*Z$ component (we fix some almost complex structure on $T^*Z$). Therefore, we have
    \begin{equation}\label{eq:moduli_LQ_L}
        \mathfrak{M}_2(L_Z, L_Z;\beta\times \pt)=\mathfrak{M}_2(L, L;\beta)\times Z
    \end{equation}
    Moreover, $\Gamma$ naturally acts on $\mathfrak{M}_2(L, L;\beta)$ and $\mathfrak{M}_2(L', L';\beta)$ so that \eqref{eq:moduli_LQ_L} is $\Gamma$-equivariant. Since the $\Gamma$-action on $Y\times T^*Z$ is free, we also have
    \begin{equation}\label{eq:moduli_LQ_L'}
    \mathfrak{M}_2(L', L';\beta\times \pt)=(\mathfrak{M}_2(L_Z, L_Z;\beta\times \pt))/\Gamma
    \end{equation}
    In particular, we have 
    \[
    n_{i(\beta)}=n_{\beta\times \pt}=n_\beta,
    \]
    which proves the first statement of (ii).

    It remains to prove part (i). It is clear from \eqref{eq:moduli_LQ_L} that $HF^\bullet(\bL_Z,\bL_Z)\cong HF^\bullet(\bL,\bL)\otimes H^\bullet(Z;\F)$, so it remains to prove that $HF^\bullet(\bL',\bL')\cong HF^\bullet(\bL_Z,\bL_Z)$. Note that \eqref{eq:moduli_LQ_L'} implies that there is a chain map
    \[
    \pr^*:CF^\bullet(\bL',\bL')\to CF^\bullet(\bL_Z,\bL_Z).
    \]
    We will need to show that $\pr^*$ is a quasi-isomorphism. By considering the spectral sequence associated to the energy filtration, it suffices to show that the classical pullback (denoted by the same symbol)
    \[
    \pr^*:H^\bullet(\bL';\F)\to H^\bullet(L_Z;\F)
    \]
    is an isomorphism. In general, $\pr^*$ gives an isomorphism
    \[
    \pr^*:H^\bullet(\bL';\F)\to H^\bullet(L_Z;\F)^\Gamma,
    \]
    so we need to show that $\Gamma$ acts trivially on $H^\bullet(L_Z;\F)$. This follows from the fact that the $\Gamma$-action on $L_Z$ extends to the action of the connected group $T\times Q'\supset \Gamma$. The proof is now complete.
\end{proof}

\section{Proofs of Theorems~\ref{thm1} and \ref{thm2}}\label{section:proof_of_theorems}
\begin{proof}[Proof of \Cref{thm1}] 
Let $s\in W$ be a reflection. We need to show that both $\mu_T(L)$ and $\HH_T(E)$ are fixed by $s$. First note that, since $G$ is connected, we have $HF^\bullet(\bL,s(\bL))\cong HF^\bullet(\bL,\bL)\neq 0$.

Now suppose $s(\mu_T(L))\neq\mu_T(L)$. Then, since
\begin{equation}\label{momentinvariance}
s(\mu_T(L))=\mu_T(s(L))\neq \mu_T(L),
\end{equation}
we have $L\cap s(L)=\varnothing$. This contradicts the assumption that $HF^\bullet(\bL,s(\bL))\neq 0$. Note that we also show that $L\cap s(L)\neq \varnothing$.

Next, suppose $s(\HH_T(E))\neq \HH_T(E)$. We may assume the $G$-action on $Y$ is free. Indeed, choose a
compact simply connected smooth manifold $Z$ equipped with a free
$G$-action, and replace
\[
Y
\quad\text{and}\quad
\bL=(L,E)
\]
by
\[
Y\times T^*Z
\quad\text{and}\quad
\bL_Z
=
\bigl(L\times Z,\operatorname{pr}_L^*E\bigr),
\]
respectively. We need the following lemma.

\begin{lem}\label{LemClean}
    There exists a $T$-invariant Lagrangian brane $\bL'=(L',E')$ such that $\bL'$ is $T$-equivariantly Hamiltonian isotopic to $s(\bL)$, and such that the intersection of $L$ and $L'$ is clean and consists of finitely many $T$-orbits.
\end{lem}
\begin{proof}
Let $c=\mu_T(L)=\mu_T(s(L))$. Then
\[
\overline L_1=L/T,\qquad \overline L_2=s(L)/T
\]
are Lagrangian submanifolds of $Y/\!/_c Y=\mu_T^{-1}(c)/T$. By Hamiltonian transversality, after an arbitrarily small Hamiltonian isotopy of $\overline L_2$, the submanifolds $\overline L_1$ and $\overline L_2$ intersect transversely, and hence in finitely many points. Let $\overline h_t$ be a Hamiltonian function generating such a Hamiltonian isotopy.

Let $\pr:\mu_T^{-1}(c)\to Y/\!/_c Y$ be the quotient map, and choose a smooth $T$-invariant function $h_t:Y\to \R$ that agrees with $\overline h_t\circ\pr$ on $\mu^{-1}(c)$. It is straightforward to check that $h_t$ generates a Hamiltonian isotopy which gives the desired $\bL'$.
\end{proof}
Replacing $s(\bL)$ with $\bL'$ if necessary, we may assume that $L$ and $s(L)$ intersect cleanly along finitely many $T$-orbits. By a spectral sequence argument in Floer theory (see (\ref{inequality}) in \Cref{appendix:Floer_theory}), 
\begin{equation}\label{H=0}
   HF^\bullet(\bL,s(\bL))\neq 0 \implies  H^\bullet(L\cap s(L),\Hom(E,s(E))\otimes \Theta)\neq 0,
\end{equation}
where $\Theta$ is a certain $\Z_2$-local system on $L\cap s(L)$, depending on the relative spin structure of $L$. 

We now show that the rank-one local system $\Hom(E,s(E))\otimes \Theta$ is nontrivial. Let $p\in L\cap s(L)$. We will show in \Cref{appendix:Floer_theory} that $\hol^p_{T}(\Theta)=1$. Therefore,
\[
\hol^p_{T}(\Hom(E,s(E))\otimes \Theta)=\hol^p_{T}(E)^{-1}\hol^p_{T}(s(E))=\HH_{T}(E)^{-1}s(\HH_{T}(E))\neq 1.
\]
The last equality follows from
\begin{equation*}
    \HH_{T}(E)=\hol^p_{T}(E)\exp(-i\kappa_{T}(p)),
\end{equation*}
and
\begin{equation*}
    \HH_{T}(s(E))=\hol^p_{T}(s(E))\exp(-i\kappa_{T}(p)).
\end{equation*}
\Cref{topological_lemma} below contradicts \eqref{H=0}. Hence, we must have $s(\HH_T(E))=\HH_T(E)$, which proves \Cref{lem:thm1}. \Cref{thm1} then follows from \Cref{prop:reduction_to_no_sp(n)}.

\begin{lem}\label{topological_lemma}
    Let $\mathcal{L}$ be a nontrivial rank-one local system on a compact torus $T$. Then $H^k(T,\mathcal{L})=0$ for all $k$.
\end{lem}
\begin{proof}
    We prove this by induction on $\dim T$. Suppose $\dim T=1$. Then $H^0(T,\mathcal{L})=0$ because $\mathcal{L}$ is nontrivial. Also, $H^k(T,\mathcal{L})$ vanishes for $k>1$ for dimension reasons, and it vanishes for $k=1$ because the Euler characteristic of $T$ is zero.

    In the general case, if $\mathcal{L}$ is nontrivial, the monodromy of $\mathcal{L}$ is nontrivial in the direction of a certain $\operatorname{U}(1)\subset T$. Let $p:T\to T/\operatorname{U}(1)$ be the quotient map. Then we have
    \[
    Rp_{*}\mathcal{L}=0
    \]
    by the $k=1$ case. The result now follows.
\end{proof}
This completes the proof of \Cref{thm1}.
\end{proof}

We need some preparation before proving \Cref{thm2}. Let $Z$ be a
smooth simply connected manifold equipped with a free $G$-action.
Then $T^*Z$ is naturally a Hamiltonian $G$-manifold. We fix a compatible almost complex structure on $T^*Z$. Let $Y'=Y\times T^*Z$, $L'=L\times Z\subset Y\times T^*Z$. Moreover, let $T_2\subset T$ be a subtorus, and write $c=\mu^Y_{T_2}(L)$, $Y''=(\mu^{Y'}_{T_2})^{-1}(c)/T_2$, and $L''=L'/T_2\subset Y''$. Let $\pi:(\mu^{Y'}_{T_2})^{-1}(c)\to Y''$ be the projection map.
\begin{lem}\label{LemTquotient}
There is a smooth deformation retract $r:Y'\to (\mu^{Y'}_{T_2})^{-1}(c)$. Moreover, if for each $\beta\in \pi_2(Y, L)$, we write $\beta''\in \pi_2(Y'',L'')$ for the image of $\beta$ under the composition
\[Y\cong Y\times \pt\subset Y'\xrightarrow{\pi\circ r}Y'',\]
then
     \begin{enumerate}[(a)]
        \item The assignment $\beta\mapsto\beta''$ gives a bijection $\iota: \pi_2^{\text{eff}}(Y, L) \to \pi_2^{\text{eff}}(Y'',L'')$ between effective disc classes that preserves the Maslov index and symplectic area. 
        \item For each $\beta\in \pi_2^{\text{eff}}(Y, L)$, $n_{\beta''}=n_\beta$.
        \item For each $\beta\in \pi_2^{\text{eff}}(Y, L)$, we have a commutative diagram
        \begin{equation*}
            \begin{tikzcd}
                \mathfrak{M}_2(L'',\beta'')\ar[r]\ar[d,"\operatorname{ev}_j"]&Z/T_2\ar[d,equal]\\
                L''\ar[r]&Z/T_2
            \end{tikzcd},
        \end{equation*}
        where $\mathfrak{M}_2(L'',\beta'')$ is the moduli space of holomorphic discs $u:\D\to Y''$ with two boundary marked points, representing the class $\beta''$, and $\operatorname{ev}_j$ for $j=0,1$ are the evaluation maps at the two marked points.
    \end{enumerate}
\end{lem}
\begin{proof}[Proof of \Cref{LemTquotient}]
Let $V_1\subset T^*Z$ be the conormal bundle to the foliation of
$Z$ by $T_2$-orbits, and choose a complementary vector subbundle
$V_2$ so that
\[
T^*Z=V_1\oplus V_2.
\]
Note that the projection $\pr:Y'=Y\times T^*Z\to Y\times V_1$ restricts to a diffeomorphism $\chi :(\mu^{Y'}_{T_2})^{-1}(c)\to Y\times V_1$. The retraction $r:Y'\to (\mu^{Y'}_{T_2})^{-1}(c)$ is defined by $r=\chi^{-1}\circ \pr$. $r$ is a deformation retract because $\chi^{-1}$ is a section for the vector bundle $Y'\to Y\times V_1$, see the proof of Proposition 4.7 in \cite{Cazassus} for more details.

Parts (a), (b), and (c) of \Cref{LemTquotient} follow from \cite[Lemma~3.35, Corollary~3.36 and Corollary~3.37]{LLL} (see also \cite{KLZ} and \cite{Cazassus}).
\end{proof}

Let $\psi:H_1(L;\F)\to H_1(L'';\F)$ be the pushforward via the composition $L\to L\times \pt\subset L'\to L''$. We now consider the case when $G=G_1\times T_2$ for a compact connected Lie group $G_1$. Let $T=T_1\times T_2$ be the induced decomposition.
\begin{prop}\label{prop:equality_df_fz}
There exist a $G_1$-equivariant B-field $\tilde B''$ on $Y''$ and a Lagrangian brane $\bL''=(L'',E'')$ such that
\[
df(\bL'')=\psi(df(\bL))
\qquad\text{and}\qquad
\fz_{T_1}(\bL)=\fz_{T_1}(\bL'').
\]
\end{prop}
\begin{proof}

Let $B'$ (resp. $\kappa'$) be the pullback of $B$ (resp. $\kappa$) via the projection $Y'\to Y$, and $E'$ be the pullback of $E$ via the projection $L'\to L$. Since $\kappa'$ is $T_2$-invariant, it also descends to a $G_1$-equivariant map $\kappa'':Y''\to \lig^*$.

It is clear that \begin{equation}\label{B'eq}
    \iota_{\xi^\sharp}B'=d\langle \kappa',\xi\rangle
\end{equation}
for all $\xi\in \lig$. We can rephrase this statement as follows. Let $\theta\in \Omega^1(Y')\otimes \lig$ be a connection form for the principal $G$-bundle $Y'\to Y'/G$, then
\[
B'+d\langle \kappa',\theta\rangle
\]
is a closed basic form with respect to the $G$-action. We write 
\begin{equation}\label{conndecomp}
    \theta=\theta_{G_1}+\theta_{T_2}
\end{equation} 
according to the decomposition $\lig=\lig_1\oplus \lt_2$. Then $B'+d\langle \kappa_{T_2}',\theta_{T_2}\rangle$ is basic with respect to the $T_2$-action. We denote the projections $Y'\to Y'/T_2$, $(\mu^{Y'}_{T_2})^{-1}(c)\to Y''$ and $L'\to L''$ by the same letter $\pi$. There exists a closed two-form $B''_0$ such that (\cite{SUSYeq})
\begin{equation}\label{basicEq}
    \pi^* B''_0=B'+d\left(\langle \kappa_{T_2}',\theta_{T_2}\rangle\right).
\end{equation}
On the other hand, let $\nabla_{E'}=d+i\gamma'$. We can write
\begin{equation}\label{basicconn}
    \gamma'=\pi^*\gamma''+\langle \pi^*\tau, \theta_{T_2}|_{L'}\rangle.
\end{equation}
where $\gamma''$ is a $T_1$-invariant one-form on $L''$ and $\tau:L''\to \lt_2^*$ is a $G_1$-invariant function. The condition $d\gamma'=B'|_{L'}$ now becomes
\begin{align}
    B_0''|_{L''}&=d\gamma''+\langle\kappa'_{T_2}+\tau,\Omega|_{L''}\rangle\label{Beq1}\\
     d\kappa''_{T_2}|_{L''}&=-d\tau\label{Beq2}.
\end{align}
Here, $\Omega$ is the curvature form for the principal $T_2$-bundle $Y'\to Y'/T_2$, i.e. $\pi^*\Omega=d\theta_{T_2}$. \Cref{Beq2} implies that $\kappa''_{T_2}+\tau$ is equal to some constant $C\in \lt_2^*$ on $L''$. We now define $B''$ and $\nabla_{E''}$ by
\begin{align}
    B''&=B''_0-\langle C,\Omega\rangle\label{brane1}\\
    \nabla_{E''}&=d+i\gamma''\label{brane2}.
\end{align}

\begin{lem}\label{checkings}
The following statements are true.
    \begin{enumerate}[(i)]
    \item $B''$ is closed.
    \item $ \iota_{\xi^\sharp}B''=d\langle \kappa_{G_1}'',\xi\rangle$ for $\xi\in \lig_1$.
    \item $d\gamma''=B''|_{L''}$.
    \item If $\beta\in \pi_2(Y,L)$, and $\beta''=(\pi\circ r)_*(\beta\times\pt)\in \pi_2(Y'',L'')$ where $r$ is the deformation retract given in \Cref{LemTquotient}, then
    \[e^{\int_{\beta} iB}\hol_{\partial \beta}(E)=e^{\int_{\beta''} iB''}\hol_{\partial \beta''}(E'').\]
    \item Let $p'=(p,q)\in L'=L\times Z$, and $p''=\pi(p,q)\in L''$. Then $\mu_{T_1}(L)=\mu_{T_1}(L'')$ and
    \[\exp(-i\kappa_{T_1}(p))\hol^p_{T_1}(E)=\exp(-i\kappa''_{T_1}(p''))\hol^{p''}_{T_1}(E'').\]
\end{enumerate}
\end{lem}
Now assume \Cref{checkings} holds. Items (i) and (ii) say that $\tilde {B}''=B''-\kappa_{G_1}''$ is an equivariant B-field. Item (iii) says that $\bL''=(L'',E'')$ is a Lagrangian brane with respect to $B''$. Recall that
\begin{align*}
  df(\bL)&=\sum_{\beta} n_{\beta} e^{i\int_{\beta} B}\hol_{\partial \beta}(E) \T^{\int_{\beta}\omega}[\partial \beta];\\
  df(\bL'')&=\sum_{\beta} n_{\beta} e^{i\int_{\beta''} B''}\hol_{\partial \beta''}(E'') \T^{\int_{\beta}\omega}[\partial \beta''].
\end{align*}
By \Cref{LemTquotient} and item (iv), it is now clear that $df(\bL'')=\psi(df(\bL))$. Moreover, item (v) is equivalent to $\fz_{T_1}(\bL)=\fz_{T_1}(\bL'')$. It remains to prove \Cref{checkings}, which follows from direct calculations below.
\begin{proof}[Proof of \Cref{checkings}]
    \begin{enumerate}[(i)]
        \item It is clear because $B''_0$ and $\Omega$ are closed, and $C$ is a constant.
        \item Let $\xi\in \lig_1$, then 
        \begin{align*}
            \pi^*\iota_{\xi^\sharp}B''&= \iota_{\xi^\sharp}(B'+d\langle\kappa'_{T_2},\theta_{T_2}\rangle-\langle C,d\theta_{T_2}\rangle)\\
            &= d\langle\pi^*\kappa''_{G_1},\xi\rangle-d\langle\kappa'_{T_2},\iota_{\xi^\sharp}\theta_{T_2}\rangle-\langle C,\iota_{\xi^\sharp}d\theta_{T_2}\rangle.
        \end{align*}
        Note that $\iota_{\xi^\sharp}\theta_{T_2}$ vanishes by the decomposition \eqref{conndecomp}. Also,
        \[
        \iota_{\xi^\sharp}d\theta_{T_2}=\mathcal L_{\xi^\sharp}\theta_{T_2}-
        d\iota_{\xi^\sharp}\theta_{T_2}=0
        \]
        because the
        $G_1$- and $T_2$-actions commute. This proves (ii).
        \item It follows from \Cref{Beq1,brane1}.
        \item Let $u:\D\to Y$ be a map representing $\beta$, $u'=u\times q:\D\to Y'$ for some $q\in Z$, and $u''=\pi\circ r\circ u':\D\to Y''$, where $r$ is the deformation retract given in \Cref{LemTquotient}. Then $u''$ represents the class $\beta''$. We have
        \begin{align*}
            e^{\int_{\beta} iB}\hol_{\partial \beta}(E)
            =&\exp \left(i\int_{u}B-i\int_{\partial u} \gamma\right)\\
            =&\exp \left(i\int_{u'}B'-i\int_{\partial u'} \gamma'\right)\\
            =&\exp \left(i\int_{u'}(\pi^*B'_0-d(\langle \kappa_{T_2}',\theta_{T_2}\rangle))-i\int_{\partial u'} (\pi^*\gamma''+\langle \pi^*\tau, \theta_{T_2}\rangle)\right)\\
            =&\exp \left(i\int_{u'}\pi^*B'_0-i\int_{\partial u'} (\pi^*\gamma''+\langle C, \theta_{T_2}\rangle)\right),
        \end{align*}
        and
        \begin{align*}
            e^{\int_{\beta''} iB''}\hol_{\partial \beta''}(E'')
            =&\exp \left(i\int_{u''}(B''_0-\langle C,\Omega\rangle)-i\int_{\partial u''} \gamma''\right)\\
            =&\exp \left(i\int_{u'}(\pi^*B'_0-\langle C,d\theta_{T_2}\rangle)-i\int_{\partial u'} \pi^*\gamma''\right)\\
            =&\exp \left(i\int_{u'}\pi^*B'_0-i\int_{\partial u'}( \pi^*\gamma''+\langle C, \theta_{T_2}\rangle)\right).
        \end{align*}
        This proves the statement. 
        \item $\mu_{T_1}(L)=\mu_{T_1}(L'')$ follows from the construction. To prove the other equality, it suffices to consider the case $T_1\cong S^1$. We then have
        \begin{align*}
            \exp(-i\kappa_{T_1}(p))\hol^p_{T_1}(E)&=\exp(-i\kappa'_{T_1}(p'))\hol^{p'}_{T_1}(E')\\
            &=\exp\left(-i\kappa'_{T_1}(p')-i\int_{T_1\cdot p'}\gamma'\right)
        \end{align*}
        Similarly,
        \begin{align*}
            \exp(-i\kappa''_{T_1}(p''))\hol^{p''}_{T_1}(E'')&=\exp\left(-i\kappa''_{T_1}(p'')-i\int_{T_1\cdot p''}\gamma''\right)\\
            &=\exp\left(-i\kappa'_{T_1}(p')-i\int_{T_1\cdot p'}\pi^*\gamma''\right).
        \end{align*}
        Using \Cref{basicconn}, it suffices to show that
        \[\int_{T_1\cdot p'}\langle \pi^*\tau, \theta_{T_2}|_{L'}\rangle=0.\]
        However, this is a consequence of the decomposition (\ref{conndecomp}).
    \end{enumerate}
\end{proof}
This finishes the proof of \Cref{prop:equality_df_fz}.
\end{proof}

\begin{proof}[Proof of \Cref{thm2}]We will perform an induction on the rank $r$ of $G$. 

If $r=1$, then $\ker(\alpha)=0$, and hence $df(\bL)=0$. \Cref{H1gen} implies that $HF^\bullet(\bL,\bL)\neq 0$, and the result follows from \Cref{thm1}.

Now suppose $\rank G=r>1$, and assume the theorem holds for any group of smaller rank. Let $\rho:\operatorname{SU}(2)\to G$ be the homomorphism determined by $\alpha$, and $T_1\cong S^1$ be the subgroup of diagonal matrices in $\operatorname{SU}(2)$. Let $T_2\subset T$ be the identity component of $T^{s_\alpha}$. We can replace $G$ with $\operatorname{SU}(2)\times T_2$ safely without affecting the assumptions and conclusion of \Cref{thm2}. We remark that the condition $\fz_T(\bL)\in \ker(\alpha^\vee)$ is equivalent to $\fz_{T_1}(\bL)=1$.

If the restriction of $o_*$ to $H_1(T_2;\F)$ is not injective, then we can find a subtorus $T'\subset T_2$ such that $o_*(H_1(T';\F))=o_*(H_1(T_2;\F))$. Therefore, we can replace $G$ with $\operatorname{SU}(2)\times T'$, and the result follows from induction.

So we may assume the restriction of $o_*$ to $H_1(T_2;\F)$ is injective. Let $Z$ be a smooth simply connected manifold equipped with a free $G$-action. The product $Y'=Y\times T^*Z$ is equipped with the diagonal Hamiltonian $G$-action, and $Y''=Y'//_cT_2$, where $c=\mu_{T_2}(L)$, is equipped with a Hamiltonian $\operatorname{SU}(2)$-action. Then $L'=L\times Z\subset Y'$ is a $T$-invariant Lagrangian, and $L''=L'/T_2\subset Y''$ is a $T_1$-invariant Lagrangian. 

By \Cref{prop:equality_df_fz}, there exists a Lagrangian brane $\bL''=(L'',E'')$ so that $df(\bL'')=\psi(df(\bL))$ and $\fz_{T_1}(\bL)=\fz_{T_1}(\bL'')$. Since the composition
\[H_1(T_2;\F)\xrightarrow{o_*}
H_1(L;\F)\xrightarrow{\psi} H_1(L'';\F)\]
is zero, the assumption $df(\bL)\in o_*(H_1(T_2;\F))=o_*(\ker(\alpha))$ implies that $df(\bL'')=0$.

We now choose $Z$ so that $\pi_k(Z)=0$ for $k\leq \dim L$. We will show that this implies $HF^\bullet(\bL'',\bL'')\neq 0$. We compute $HF^\bullet(\bL'',\bL'')$ using the canonical model $(H^\bullet(L'';\F),\{m_k\}_{k\geq 0})$. 

We will show that the unit
\[
1\in H^0(L'';\mathbb F)
\]
defines a nonzero class in
\[
HF^\bullet(\bL'',\bL'').
\]
Let $x\in H^k(L'';\F)$ and $\beta\in \pi_2(Y, L)$. It suffices to show that $m_{1,\beta''}(x)$ has no degree 0 component, or equivalently, $\langle m_{1,\beta''}(x),\pt\rangle=0$.

Suppose first $k>\dim L$. By \Cref{LemTquotient} we have the following commutative diagram for $j=0,1$.
        \begin{equation*}
            \begin{tikzcd}
                &\mathfrak{M}_2(L'',\beta'')\ar[r]\ar[d,"\operatorname{ev}_j"]&Z/T_2\ar[d,equal]\\
               L\ar[r] &L''\ar[r,"\pr"]&Z/T_2.
            \end{tikzcd}
        \end{equation*}
Let $S$ be a cycle Poincar\'e dual of $x$. Since $x\in H^{>\dim L}$, $\dim S<\dim Z/T_2$ and $\pr(S)$ is a proper subset of $Z/T_2$. Let $p''\in L''$ be a point so that $\pr(p'')\not \in \pr(S)$. As a result $\operatorname{ev}_0^{-1}(S)\cap \operatorname{ev}_1^{-1}(p'')=\varnothing$, and hence $\langle m_{1,\beta''}(x),\pt\rangle=0$.

We next claim that $H^{\leq \dim L}(L'';\F)$ is generated by $H^1(L'';\F)$. Our assumption on $Z$ implies $H^k(L'';\F)=H^k_{T_2}(L;\F)$ for $k\leq \dim L$, so it suffices to prove that $H^{\bullet}_{T_2}(L;\F)$ is generated by $H^1_{T_2}(L;\F)$. This follows from the decomposition
\begin{equation}\label{eq:decomposition_H}
    H^\bullet(L;\F)\cong H^\bullet_{T_2}(L;\F)\otimes H^\bullet(T_2;\F)
\end{equation}
and the assumption that $H^\bullet(L;\F)$ is generated in degree 1.

We now prove \eqref{eq:decomposition_H}. By induction on $\operatorname{rank}T_2$, we may assume $T_2\cong S^1$. Since $o_*$ is injective, there exists an $S^1$-invariant closed one-form $\theta$ whose restriction to each $S^1$-orbit represents
the generator of $H^1(S^1,\R)$. Then $i\theta$ is a
connection one-form for the principal $S^1$-bundle
\[
L\longrightarrow L/S^1.
\]
Its curvature vanishes, and~\eqref{eq:decomposition_H} follows from the
Leray--Hirsch theorem.

By \Cref{H1gen}, $m_1=0$ on $H^{\leq \dim L}(L'';\F)$. It now follows that $HF^\bullet(\bL'',\bL'')\neq 0$. Therefore, \Cref{thm1} implies $\fz_{T_1}(\bL)=\fz_{T_1}(\bL'')\in Z(\operatorname{SO}(3,\C))=\{1\}$, and the theorem is proved.
\end{proof}

\appendix
\section{Lagrangian Floer cohomology}\label{appendix:Floer_theory}
For a pair of compact connected relatively spin Lagrangian submanifolds $L_1,L_2\subset Y$, one can define their Lagrangian Floer cohomology; see \Cite{FOOOI,Seidelbook,BiranOctav}. When $L_1=L_2=L$, the Floer cochain complex carries a filtered
$A_\infty$-algebra structure. We follow the approach of \cite{FOOOI}. The generalization incorporating a $B$-field is discussed in \cite{FukB,CHO}.

For each connected component $C$ of $L_1\cap L_2$, we let $\Omega^\bullet(C,\Hom(E_1|_C,E_2|_C))$ be any standard cochain model (e.g. de Rham, Morse, etc., see \Cite{FOOOI,Kuranishi,KLZ,BiranOctav,LLL}) that computes the cohomology group $H^\bullet(C,\Hom(E_1|_C,E_2|_C))$. We define
\[CF^\bullet(\bL_1,\bL_2)=\bigoplus_{C\in \pi_0(L_1\cap L_2)}\Omega^\bullet(C,\Hom(E_1|_C,E_2|_C)\otimes \Theta)\otimes \F.\]
Here, $\Theta$ is a $\mathbb Z_2$-local system on
$L_1\cap L_2$ that depends on the relative spin structures of
$L_1$ and $L_2$; see \cite[Section~8.8]{FOOOII}. Since we assume $T$ preserves $\omega, J$, as well as the relative spin structures on $L_1$ and $L_2$, it follows immediately from the construction in \emph{loc.\ cit.} that $\Theta$ is $T$-equivariant.

\begin{lem}\label{ThetaLemma}
    If $T=S^1$ and the intersection $L_1\cap L_2$ is clean, then $\hol_{S^1}(\Theta)=1$ on each connected component $C$ of $L_1\cap L_2$. 
\end{lem}
\begin{proof}
    Let $f:S^1\to S^1\cdot p$ be any orbit map. Then $f^*\Theta$ is an $S^1$-equivariant local system on $S^1$, hence must be trivial.
\end{proof}

To define the Floer differential, let $C, C'\in \pi_0(L_1\cap L_2)$. We consider $J$-holomorphic maps
\[
u:\Sigma\longrightarrow Y
\]
from a genus-zero bordered Riemann surface $\Sigma$ with two
boundary marked points $z,z'$ such that:
\begin{enumerate}
\item $u(z)\in C$ and $u(z')\in C'$;
\item the clockwise boundary arc from $z'$ to $z$ is mapped to
$L_1$, while the complementary boundary arc from $z$ to $z'$ is
mapped to $L_2$.
\end{enumerate}
Each such $u:\Sigma\to Y$ defines a parallel transport \[P_u:\Hom(E_1|_{u(z)},E_2|_{u(z)})\to \Hom(E_1|_{u(z')},E_2|_{u(z')})\]
given by $\ell\mapsto P_{E_2}([z,z'])\circ\ell\circ P_{E_1}([z',z])$, where $P_{E_1}([z',z])$ is the parallel transport of $E_1$ along the arc $u([z',z])$ and similarly for $P_{E_2}([z,z'])$. 

Let $\mathfrak{M}_2(L_1,L_2;C,C')$ be the moduli space of all such maps $u$, defined using the theory of Kuranishi structures \cite{Kuranishi}.  

The connected components of $\mathfrak{M}_2(L_1, L_2;C,C')$ are labeled by the homotopy classes of the map $u$. Let $\beta$ be such a homotopy class, we write $\mathfrak{M}_2(L_1, L_2;C,C';\beta)$ (or $\mathfrak{M}_2(L_1, L_2;\beta)$ if $C$ and $C'$ are obvious) for the corresponding component, whose expected dimension is equal to $\frac{1}{2}\dim Y+\mu(\beta)-1$ where $\mu(\beta)$ is the Maslov index of $\beta$. Consider the diagram
\begin{equation}\label{pullpush}
    \begin{tikzcd}
    &\mathfrak{M}_2(L_1,L_2;C,C';\beta)\ar[ld,"{ev_{z}}" swap]\ar[rd,"ev_{z'}"]\\
    C&&C',
\end{tikzcd}
\end{equation}
where $ev_{z}$ and $ev_{z'}$ are the evaluation maps at $z$ and $z'$, respectively. 
Let 
\[\delta_{L_1,L_2,C,C',\beta}:\Omega^\bullet(C,\Hom(E_1|_C,E_2|_C)\otimes \Theta)\otimes \F\to \Omega^\bullet(C',\Hom(E_1|_{C'},E_2|_{C'})\otimes \Theta)\otimes \F\]
be the linear map obtained from pull-push using diagram (\ref{pullpush}) and twisted by $P_\beta$. Explicitly, 
\[\delta_{L_1,L_2,C,C',\beta}=(ev_{z'})_*\circ P_\beta\circ ev^*_{z}\]
where $P_\beta: \Hom(ev_{z}^*E_1,ev_{z}^*E_2)\to \Hom(ev_{z'}^*E_1,ev_{z'}^*E_2)$ is the holonomy induced by $P_u$ defined above.
Now the Floer differential
\begin{equation*}
   \delta=\delta_{L_1,L_2}:CF^\bullet(\bL_1,\bL_2)\to CF^\bullet(\bL_1,\bL_2)
\end{equation*}
 is defined by the formula
 \begin{equation*}
     \delta=\sum_{C,C',\beta}\exp\left(i\int_\beta B\right)\T^{\int_\beta\omega}\delta_{L_1,L_2,C,C',\beta}.
 \end{equation*}
 When $\delta^2=0$, the cohomology of this complex is denoted by $HF^\bullet(\bL_1,\bL_2)$.

 We recall the following definition (\cite{FukB}). 
 \begin{df}
 We say two Lagrangian branes $\bL_1=(L_1,E_1)$ and $\bL_2=(L_2,E_2)$ are Hamiltonian equivalent if there exists 
 \begin{enumerate}
\item a Hamiltonian isotopy
\[
\psi_t:Y\longrightarrow Y,
\qquad t\in[0,1],
\]
with $\psi_0=\operatorname{id}$ and $\psi_1(L_1)=L_2$;

\item a unitary line bundle $\mathcal E\to L_1\times[0,1]$
equipped with a connection $\nabla^{\mathcal E}$ such that, for
each $t\in[0,1]$,
\[
F_{\nabla^{\mathcal E}|_{L_1\times\{t\}}}
=
i(\psi_t|_{L_1})^*B;
\]

\item connection-preserving isomorphisms
\[
(\mathcal E,\nabla^{\mathcal E})|_{L_1\times\{0\}}
\cong
(E_1,\nabla_{E_1})
\]
and
\[
(\mathcal E,\nabla^{\mathcal E})|_{L_1\times\{1\}}
\cong
(\psi_1|_{L_1})^*(E_2,\nabla_{E_2}).
\]
\end{enumerate}
If the Hamiltonian isotopy and the line-bundle data are
$T$-equivariant, we say that $\underline L_1$ and
$\underline L_2$ are $T$-equivariantly Hamiltonian equivalent.
 \end{df}
 
 In this paper, we consider only the case when $\bL_1$ is weakly unobstructed and $\bL_2$ is Hamiltonian equivalent to $\bL_1$. This implies $\delta^2=0$ (\cite{FOOOI}). Moreover, if $\bL'_2$ is Hamiltonian equivalent to $\bL_2$, then
\begin{equation}\label{HamIso}
    HF^\bullet(\bL_1,\bL_2)\cong HF^\bullet(\bL_1,\bL_2').
\end{equation}
This is \cite[Theorem~G.4]{FOOOI}. 

We also need the following inequality:
\begin{equation}\label{inequality}
\dim_{\F}HF^\bullet(\bL_1,\bL_2)\leq\dim_{\F}H^\bullet(L_1\cap L_2;\Hom(E_1,E_2)\otimes \Theta\otimes \F).
\end{equation}
This is \cite[Theorem~G.2]{FOOOI}. It follows from a spectral
sequence converging to Floer cohomology whose $E_2$-page is given
by the classical cohomology, with coefficients in
\[
\operatorname{Hom}(E_1,E_2)\otimes\Theta,
\]
tensored with the appropriate associated graded Novikov ring; see
also \cite[Theorem~6.1.4]{FOOOII}. In particular, if the classical cohomology with coefficients in $\operatorname{Hom}(E_1,E_2)\otimes\Theta$ vanishes, then the Floer cohomology also vanishes.

It is known that the Floer cochain complex $CF^\bullet(\bL,\bL)$ carries a filtered $A_\infty$-algebra structure that is quasi-isomorphic to its canonical model 
\[
(H^\bullet(L),\{m_k\}_{k\geq 0});
\]
see \cite{FOOOI,Canonicalmodel}.
\begin{lem}\label{H1gen}
    Assume the minimal Maslov number of $L$ is at least $2$. Let $k$ be a non-negative integer. Suppose $df(\bL)=0$ and $H^{\leq k}(L;\F)$ is contained in the subalgebra (with respect to the wedge product) of $H^\bullet(L;\F)$ generated by $H^1(L;\F)$. Then $m_1$ vanishes on $H^{\leq k}(L;\F)$. In particular, if $H^\bullet(L;\F)$ is generated in degree 1, then $m_1=0$.
\end{lem}
\begin{proof}
We prove the lemma by induction on $k$. The case $k=0$ is obvious, and the case $k=1$ follows from the definition, see the proof of Proposition 6.9 in \cite{AurouxComp}. 

Suppose $k>1$, and the conclusion is known for smaller $k$. Let $z\in H^k(L;\F)$. We can write $z=\sum_i x_i\wedge y_i$, where $x_i\in H^1(L;\F)$ and $y_i\in H^{k-1}(L;\F)$. We have
    \[\sum_im_2(x_i,y_i)=z+\sum_{i,\beta}m_{2,\beta}(x_i,y_i),\]
    where the sum is over all non-constant effective disc classes $\beta\in \pi_2(Y,L)$.     

    By the induction hypothesis,
    \[
    m_1\bigl(m_{2,\beta}(x_i,y_i)\bigr)=0
    \]
    for every nonconstant effective class $\beta$, since $m_{2,\beta}(x_i,y_i)$ has degree $k-\mu(\beta)<k$. Moreover, the $A_\infty$-relations give
    \[
    m_1\left(\sum_i m_2(x_i,y_i)\right)=\sum_i\left(m_2(m_1(x_i),y_i)\pm m_2(x_i,m_1(y_i))\right)=0,
    \]
    where the terms involving $m_0$ vanish by weak unobstructedness and strict unitality. It follows that $m_1(z)=0$. 
\end{proof}

\printbibliography
\end{document}